\documentclass[12pt,twoside, a4paper, english, reqno]{amsart}
\usepackage[dvips]{epsfig}
\usepackage{amscd}
\usepackage{amssymb}
\usepackage{amsthm}
\usepackage{amsmath}
\usepackage{latexsym}
\usepackage{amsaddr}

\usepackage{upref}
\usepackage{hyperref}

\hypersetup{linkcolor=blue, colorlinks=true
,citecolor = red}
\usepackage{color}
\theoremstyle{plain}
\newtheorem{thm}{Theorem}[section]

\theoremstyle{plain}
\newtheorem{lem}[thm]{Lemma}

\theoremstyle{definition}
\newtheorem{defi}{Definition}[section]
\newtheorem*{rem}{Remark}

\newcommand{\R}{\mathbb{R}}
\renewcommand{\S}{\mathbb{S}}

\newcommand{\rn}{\mathbb{R}^{N}}

\newcommand{\rnn}{\mathbb{R}^{2}}
\newcommand{\hn}{\mathbb{H}^{N}}
\newcommand{\hnn}{\mathbb{H}^{2}}

\newcommand{\authorfootnotes}{\renewcommand\thefootnote{\@fnsymbol\c@footnote}}%
\newcommand{\hes}{\rm{Hess}}
\newcommand{\Exp}{ {\rm exp}}

\def\e{{\text{e}}}

\usepackage{hyperref}
\numberwithin{equation}{section} \allowdisplaybreaks

\title{Adams Inequality on pinched Hadamard Manifolds}
\author{Jerome Bertrand$^\dagger$ and   kunnath Sandeep$^{\dagger\dagger}$}
\thanks{$^{\dagger}$ Institut de Math\'ematiques de Toulouse, UMR CNRS 5219 Universit\'e Toulouse III, Toulouse Cedex 9, France. E-mail : bertrand@math.univ-toulouse.fr\\
$^{\dagger\dagger}$ TIFR  Centre for Applicable Mathematics, Post Bag No. 6503,Sharadanagar,Yelahanka New Town, Bangalore 560065. Email: sandeep@math.tifrbng.res.in}
\begin{document}
\footnotetext{This work is supported by the IFCAM project "Functional Inequalities PDE's and Geometry."}

\begin{abstract}
In this article we prove the Adams type inequality for $W^{k,p}(M)$ functions, where $(M,g)$ is a $n$-dimensional Hadamard manifold with sectional curvature bounded from below and above by a negative constant and $k$ is an integer satisfying $kp=n$. 
\end{abstract}

\maketitle
MSC2010 Classification: {\em 46E35, 58E35.}\\
Keywords: {\em Adams Inequality, Hadamard manifolds.}
\section{Introduction}In this article we focus on the Adams inequality on Hadamard manifolds. Recall a Hadamard manifold is a complete simply connected manifold of nonpositive sectional curvature and Adams inequalities are the optimal Sobolev embedding of the Sobolev space $W^{k,p}$ when $kp=n$, where $n$ is the dimension of the space.\\\\
There are many works on Sobolev embeddings on Riemannian manifolds and we know in particular that the Sobolev embedding holds when the manifold is compact. To be precise, let $(M,g)$ be a compact Riemannian manifold then the Sobolev embedding states that the Sobolev space $W^{k,p}(M)$ is continuously embedded into $L^q(M)$ where $q= \frac{np}{n-kp}$ provided $1\le p < \frac{n}{k}.$
The precise inequalities with precise constants describing these embeddings are of importance in both partial differential equations and geometric analysis, the study of these inequalities has been a hot topic of research for the past many decades . However when $M$ is a complete noncompact manifold then the Sobolev embedding is a nontrivial issue. In fact there exists a complete noncompact Riemannian manifold $M$ for which the Sobolev embedding $W^{k,p}(M)\hookrightarrow L^q(M)$ does not hold for any $p$ satisfying $kp<n$, where $q=\frac{np}{n-kp}$.We refer to \cite{Hebey} for a detailed discussion on the topic. \\\\When $M$ is compact and $p=\frac{n}{k}$, one can easily see that $W^{k,p}(M)$ is continuously embedded into $L^q(M)$ for all $q < \infty$ but not for $q=\infty$ and hence none of the above embeddings $W^{k,p}(M) \hookrightarrow L^q(M)$, for $q<\infty$, are optimal.  When $M$ coincides with a bounded domain $\Omega$ in $\R^n$ with smooth boundary and $k=1$, an embedding  of the Sobolev space $W^{1,p}_0(\Omega)$ into an Orlicz space establishing the exponential integrability of these functions was obtained by  Poho\v{z}aev \cite{Pohozaev} and Trudinger \cite{Trudinger}. In 1971, J.Moser \cite{Moser} while trying to study the question of prescribing the Gaussian curvature on the sphere understood the need for establishing a sharp form of the embedding obtained by Poho\v{z}aev and Trudinger. He showed that there exists a positive constant $C_0$ depending only on $n$ such that
\begin{align} \label{Moser}
\sup_{u \in C^{\infty}_c(\Omega), \int_{\Omega} |\nabla u|^n \leq 1} 
\int_{\Omega} e^{\alpha |u|^{\frac{n}{n - 1}}} \ dx \leq C_0 |\Omega|
\end{align}
holds for all $\alpha \le \alpha_n = n\left[\omega_{n-1}\right]^\frac{1}{n-1} $, where $\Omega$  is a bounded domain in $\R^n,$, $|\Omega|$ denotes the volume 
of $\Omega$, and $\omega_{n-1}$ denotes the $(n-1)$-dimensional volume of the sphere $\S^{n-1}$. Moreover when $\alpha > \alpha_n$, the above supremum is infinite. Moser, in the same paper, established the appropriate version of this sharp inequality on the sphere $\S^{2}$ and later Cherrier \cite{Che} proved it for a non-optimal exponent on any compact Riemannian manifold. These optimal inequalities of the Sobolev space $W^{1,n}(M)$, where $n$ is the dimension of $M$, are called the Moser-Trudinger inequalities.\\\\
Even though one expects a similar type inequality to hold for higher order Sobolev spaces, it is not at all obvious how to modify the proofs of the case $k=1$ to $k>1$ due to the failure of Polya-Szego type inequalities for higher order gradients $\nabla^k$. In a significant work, D.R. Adams \cite{A} established the sharp embedding in the case of higher order Sobolev spaces $W^{k,p}_0(\Omega)$ when $kp=n$. He found the sharp constant $\beta_0$ for the higher order Trudinger-Moser type inequality. More precisely, he proved that if $k$ is a positive integer less than $n,$ then there exists a constant $c_0 = c_0(k,n)$ such that 
\begin{align} \label{Adamsintegral}
 \sup_{u \in C^{k}_c(\Omega), \int_{\Omega} |\nabla^k u|^p \leq 1}\int_{\Omega} e^{\beta |u(x)|^{p^{\prime}}} \ dx \leq c_0 |\Omega|,
\end{align}
for all $\beta \leq \beta_0(k,n)$ and for all bounded domains $\Omega$ in $\R^n$, where $p = \frac{n}{k},\;p^{\prime} = \frac{p}{p-1} $,
\begin{align} \label{beta0}
 \beta_0(k,n) = 
 \begin{cases}
  \frac{n}{\omega_n} \left[\frac{\pi^{\frac{n}{2}}2^k \Gamma\left(\frac{k + 1}{2}\right)}
  {\Gamma \left(\frac{n- k + 1}{2}\right)}\right]^{p^{\prime}}, \ \ \ \mbox{if $k$ is odd}, \\ \ \
  \frac{n}{\omega_n} \left[\frac{\pi^{\frac{n}{2}}2^k \Gamma\left(\frac{k}{2}\right)}
  {\Gamma \left(\frac{n - k }{2}\right)}\right]^{p^{\prime} }, \ \ \ \mbox{if $k$ is even}, \\
  \end{cases}
\end{align}
and $ \nabla^k $ is defined by
 \begin{align}
  \nabla^k :=
  \begin{cases}
   \Delta^{\frac{k}{2}} , \ \ \ \ \ \ \ \mbox{if} \ k \ \mbox{is even}, \\
   \nabla \Delta^{\frac{k-1}{2}} , \ \ \mbox{if} \ k \ \mbox{is odd}.
  \end{cases}
\end{align}
Furthermore, if $\beta > \beta_0,$ then the supremum in \eqref{Adamsintegral} is infinite.\\\\Subsequently,  Fontana in \cite{Fonta} obtained the following  sharp version of \eqref{Adamsintegral} on compact Riemannian manifolds:\\
Let $(M,g)$ be an $n$-dimensional compact Riemannian manifold without boundary, and $k$ be a positive integer less than $n,$ then there exists a constant $c_0 = c_0(k,M)$ such that
\begin{align} \label{Fontanaintegral}
 \sup_{u \in C^{k}(M),\int_Mu =0, \int_{\Omega} |\nabla^k u|^p \leq 1}\int_{M} e^{\beta |u(x)|^{p^{\prime}}} \ dx \leq c_0
\end{align}
if $\beta \leq \beta_0(k,n),$ where $p,p^{\prime},\nabla_g^k $ are as above, and where $\nabla_g$ and $\Delta_{g}$ are the gradient and Laplace Beltrami operators with respect to the metric $g.$ Furthermore, if $\beta > \beta_0,$ then the supremum in \eqref{Fontanaintegral} is infinite.
These type of sharp inequalities satisfied by the $W^{k,p}(M)$ functions when $kp=n$ are called the Adams inequalities.\\\\
In this article, our focus will be on Adams inequalities on Hadamard manifolds.
First observe that Hadamard manifolds have infinite volume and hence $ \int_{M} e^{\beta |u(x)|^{p^{\prime}}} dx $ is infinite even for the trivial function $u=0$. To tackle these issues, we modify the exponential function and look for inequalities of the form
\begin{equation}\label{aim}
 \sup_{u \in C^{k}_c(M), \int_{M} |\nabla_g^k u|^p \leq 1}\int_{M} E_{s}({\beta |u(x)|^{p^{\prime}}}) \ d\mu_g(x)< \infty
\end{equation} 
for $\beta \leq \beta_0(k,n)$, where $\beta_0(k,n)$ is defined as in \eqref{beta0} and $E_{s}(x) = e^x-\sum\limits_{i=0}^{s-1}\frac{x^i}{i!}$ for some positive $s\in \mathbb{N}.$\\\\
First, observe that if \eqref{aim} holds for some positive $s\in \mathbb{N}$, then as a consequence we will have the inequality
\begin{equation}\label{consequnce}
\left[\int_{M}  |u(x)|^{sp^{\prime}} \ d\mu_g(x)\right]^{\frac{p}{sp^\prime}} \le C \int_{M} |\nabla_g^k u|^p \ d\mu_g(x) ,\;\; \forall u \in C^{k}_c(M).
\end{equation}
When $M$ is the Euclidean space $\R^n$, using standard scaling arguments we can see that such inequalities and hence \eqref{aim} are impossible as $kp=n$. However, in this case, one can prove embeddings if one replaces the constraint
$\int_{M} |\nabla^k u|^p \leq 1$ by $\int_{M} |\nabla^k u|^p + \lambda\int_{M} |u|^p  \leq 1$ for some positive constant $\lambda,$
see  Cao \cite{Cao}, Panda \cite{Panda}, J.M. do \'{O} \cite{DoO}, Ruf \cite{Ruf}, Li-Ruf \cite{LiRuf}, and the references therein.
\\\\
When the sectional curvature is bounded from above by a negative constant we do have inequalities like \eqref{consequnce}. For example we have the Poincare inequality which follows from Theorem \ref{mckean}. Therefore, one type of spaces where we expect Adams inequality of the form \eqref{aim} is this set of strictly negatively curved spaces. In the case of constant negative curvature, namely the hyperbolic space, Trudinger-Moser and Adams inequalities have been investigated in detail.
For $k = 1, n = 2,$ Mancini-Sandeep \cite{MS} proved the Trudinger-Moser inequality
in the hyperbolic space or, in other words, $W^{1,2}(\hnn)$ is embedded into the Zygmund space $Z_{\phi}$ determined by 
the function $ \phi = (e^{4\pi u^2} - 1).$ Another proof of this inequality was given by  Adimurthi-Tinterev \cite{AdiT}. In fact in \cite{MS}, they obtained the following general theorem:\\
 Let $\mathbb{D}$ be the unit open disc in $\rnn,$ endowed with a conformal metric
$h = \rho g_e,$  where $g_e$ denotes the Euclidean metric and $\rho \in C^2(\mathbb{D}), \rho > 0,$ then
\begin{align} \label{TMH2}
 \sup_{u \in C^{\infty}_c(\mathbb{D}), \int_{\mathbb{D}} |\nabla_h u|^2 \leq 1} \int_{\mathbb{D}} 
                     \left(e^{4\pi u^2} - 1\right) \ d\mu_h < \infty
\end{align}
holds true if and only if $h \leq c\, g_{\hnn}$ for some positive constant $c.$ Here, $\nabla_h, d\mu_h$ denotes respectively the gradient and volume element for the metric $h$, and 
$g_{\hnn} =  \left(\frac{2}{1 - |x|^2}\right)^2 \left(dx^2_1+ dx^2_2\right)$ is the Poincare metric in the disc.\\
Extensions of this inequality to $n>2$ were obtained in Lu-Tang \cite{LuT} and Battaglia-Mancini \cite{BM}. See also \cite{MST} for another proof and related issues.\\\\
Various forms of Adams inequality in the hyperbolic space were proved by  Karmakar and Sandeep \cite{KaSa} and  Fontana and Morpurgo \cite{FontaMo}. In \cite{FontaMo}, it was shown that \eqref{aim} holds when $M$ is the hyperbolic space and $k=[p-1]$, where $[x]$ denotes the smallest integer greater than or equal to $x$. In \cite{KaSa} another approach was taken from the point of view of prescribing the $Q$-curvature; the authors proved the following inequality with $p=2$:
 \begin{align}
  \sup_{u \in C^{\infty}_c(M) ,\  \int_{M}(P_{\frac{n}{2}}u)u \  d\mu_g \ \leq 1} \int_{M} \left(e^{\beta u^2} - 1\right) \ d\mu_g < +\infty
 \end{align}
iff $\beta \leq \beta_0(\frac{n}{2},n)$, where $\beta_0$ is as before and $M$ is the $n$-dimensional hyperbolic space and $P_{\frac{n}{2}}$ is the critical GJMS operator in the hyperbolic space. Related inequalities with Hardy type potentials were obtained in \cite{LLY}.\\\\
Moser-Trudinger inequality has been proved for general Hadamard manifolds in \cite{YSY}. Namely, the authors showed that when $M$ is a Hadamard manifold then for any $\lambda >0$ the inequality 
\begin{equation}\label{MT-Had}
 \sup_{u \in C^{1}_c(M), \int_{M} (|\nabla u|^n+\lambda |u|^n)\ d\mu_g \leq 1}\int_{M} E_{n-1}({\beta |u(x)|^{\frac{n}{n-1}}}) \ d\mu_g(x)< \infty
\end{equation} 
holds with the optimal choice of $\beta $ as $n\left[\omega_{n-1}\right]^\frac{1}{n-1} $.\\\\
In this article we investigate the validity of Adams inequality of the form \eqref{aim} in general pinched Hadamard manifolds. The main difficulty one faces in this task is to handle the case of infinite volume. Also, unlike in the constant curvature spaces, estimates on balls of fixed radius will depend on the center of the ball. To handle these situations we make some assumptions on the curvature. Following is the main result in this article.

\begin{thm}\label{theorem1}
Let $(M,g)$ be an $n$-dimensional pinched Hadamard manifold satisfying  $ K_g\le -a^2$ and $Ric_g \ge -(n-1)b^2 $ for some $a,b>0$\footnote{Consequently, $K_g$ is bounded from below as well.}. Let $k$ be an integer satisfying $ 1\le k   < n$ and $p = \frac{n}{k} $. Then,
\begin{align} 
 \sup_{u \in C^{k}_c(M), \int_{M}|\nabla^k u|^p \leq 1}\int_{M} E_{[p-1]}({\beta |u(x)|^{p^{\prime}}}) \ d\mu_g(x)< \infty
\end{align}
iff $\beta \leq \beta_0(k,n)$, where $\beta_0(k,n)$ is as defined in \eqref{beta0}. 
\end{thm}
As a consequence of the above theorem we can argue as in \cite{MS, KaSa} to get the exact asymptotic behaviour of the best constant of the Sobolev embedding $ W_0^{k,p}(M)\hookrightarrow L^q(M)$ as $q\rightarrow \infty.$\\
Let $(M,g) \ ,k,p $ as in Theorem \ref{theorem1}, then for any  $q \in (p,\infty) $, the following inequality holds
\begin{equation}\label{Lqinequality}
S_{q}\left[\int\limits_M |u|^q\ d\mu_g\right]^{\frac{p}{q}} \ \le \ \int\limits_M|\nabla_g^ku|^p  \ d\mu_g \ , \ \forall \ u \in C_c^\infty(M)
\end{equation} 
where $ S_q $ denotes the optimal constant in the above inequality which may depend on $n,k,q$. The above inequality easily follows from Theorem \ref{theorem1} when $q$ is of the form $ q= s\frac{p}{p-1}$ where $s$ is an integer satisfying $s\geq [p-1]$. For other values of $q$, it follows by interpolation. Then, it is obvious that 
$ \lim_{q \rightarrow \infty } S_q =0$ as otherwise it will imply embedding of $W_0^{m,p}(M)$ into $L^\infty(M)$, which is not true. We show that:
\begin{thm}\label{theorem2} Let $(M,g)$ be as in Theorem \ref{theorem1}, then
$$\lim_{q\rightarrow \infty}\left[ q^{p-1}S_q \right] \ = \ \left[\frac{p}{p-1} e \beta_0(k,n)\right]^{p-1}.$$
\end{thm}

\noindent We will establish Theorem \ref{theorem1} by converting it into an estimate on operators given by kernels, an idea initiated in this case by Adams \cite{A} and developed further in \cite{Fonta}, \cite{FonMo}, and \cite{FontaMo}. We will implement this scheme by writing the function $u$ as integral operators given by kernels. The properties of these kernels leading to Adams type inequalities with best constants have been given in \cite{FontaMo}. The real issue in our case is to establish these conditions on kernels. For instance, in order to hold true, these properties require some (locally) uniform control of the kernels in terms of the Riemannian distance between the variables. In the constant curvature case explicit formulas make this job easy, but in our case we lack these explicit formulas for kernels. Also, compared to the case of smooth compact Riemannian manifolds, where the curvature tensor and all its covariant derivatives have bounded norms, we highlight that only bounds on the second derivatives of the Riemann metric (through the sectional curvature) are actually needed in order to control the kernels. This is done by a careful analysis involving, among other things, comparison theorems from Riemannian geometry.\\\\
We divide this article into four sections. Section 2 will be devoted to preliminary materials, Section 3 will develop the details required on Green's function, and the proof of main theorems will be given in Section 4. 

\vspace{0.3cm}

{ \it Acknowledgments. The authors would like to thank Gilles Carron for useful discussions regarding Theorem 2.6.}

\section{Notation and Preliminaries} In this section we will introduce our notation and recall some results from Riemannian geometry which we will be using in this article. For more details and proofs of theorems, we refer to any standard book on Riemannian geometry like \cite{chavel, gallot, petersen}.
\subsection{Notation}
 We will denote by $(M,g)$ a Riemannian manifold with inner product $g(\cdot,\cdot)$. The Ricci and sectional curvatures will be denoted by $Ric_g$ and $ K_g$ respectively.\\\\
A Hadamard manifold is a complete simply connected Riemannian manifold $(M,g)$ with $ K_g(x) \le 0$ for all $x\in M$. We will denote the $n$-dimensional hyperbolic space of constant curvature $\lambda < 0$ by $\mathbb{H}^n_\lambda$.\\\\
The Riemannian distance between $x$ and $y$ will be denoted by $d_g(x,y)$ and the Riemannian measure will be denoted by $\mu_g.$ The Riemannian volume of the  Euclidean unit sphere $\S^{n-1}$ will be denoted by $\omega_{n-1}$.\\\\
 Let us also denote by $\nabla_g $ and $\Delta_g= + \text{Tr} \, \text{Hess}$  the gradient and the Laplace Beltrami operator associated with the metric $g$. Moreover, for a positive 
 integer $k,$ let $\Delta_g^k$ be the $k$-th iterated Laplacian, we define the $k$-th order gradient $\nabla_g^k$ by,
 \begin{align}
  \nabla_g^k :=
  \begin{cases}
   \Delta_g^{\frac{k}{2}} , \ \ \ \ \ \ \ \mbox{if} \ k \ \mbox{is even}, \\
   \nabla_g \Delta_g^{\frac{k-1}{2}} , \ \ \mbox{if} \ k \ \mbox{is odd},
  \end{cases}
\end{align} 
For $u\in C^k(M)$ and  $x\in M$, we define $|\nabla_g^ku(x)|$ as the modulus of $\Delta_g^{\frac{k}{2}}u(x)$ when $k$ is even, and $\sqrt{g\left(\nabla_g \Delta_g^{\frac{k-1}{2}}u(x),\nabla_g \Delta_g^{\frac{k-1}{2}}u(x)\right)}$ when $k$ is odd.
\subsection{Some results from Riemannian Geometry.} One of the main difficulties we will face in proving our result comes from the infinite measure of these manifolds. First, we will recall some results on the volume.\\\\
Let $V^n_\lambda(r)$ denote the volume of a ball with radius $r>0$ in the $n$-dimensional space form of constant curvature $\lambda \in (-\infty,0]$, then 
\begin{align}\label{volume-model}
  V^n_\lambda(r)=
  \begin{cases}
   \frac{\omega_{n-1}}{n}r^n , \ \ \ \ \ \ \ \ \ \ \ \ \ \mbox{if} \ \ \lambda = 0, \\
   \frac{\omega_{n-1}}{a^n}\int\limits_0^{ar}\sinh^{n-1}s\;ds , \ \ \mbox{if} \ \ \lambda = -a^2 <0
  \end{cases}
\end{align} 
In the general case, we have the Bishop-Gromov volume comparison theorem:
\begin{thm}\label{volume} Let $(M,g)$ be an $n$-dimensional complete Riemannian manifold with $Ric \ge (n-1)\lambda $ for some $\lambda \in \R$ then for any $x\in M$ the volume ratio 
 $$\frac{\mu_g(B(x,r))}{V^n_\lambda(r)} $$
 is a nonincreasing function of $r$. In particular
 $$ \frac{\mu_g(B(x,r))}{V^n_\lambda(r)} \le \lim\limits_{R\rightarrow 0}\frac{\mu_g(B(x,R))}{V^n_\lambda(R)} =1$$
and hence $\mu_g(B(x,r))\le V^n_\lambda(r). $
\end{thm}
This result follows from estimates on the volume element due to Bishop that we will also use in the following:
\begin{thm}\label{volumeelement} Let $(M,g)$ be an $n$ dimensional complete Riemannian manifold. For $x\in M$, let $r^{n-1}A_x(r,\theta)\;d\theta dr$ denotes the Riemannian measure in normal coordinates centered at $x$.

Let us first assume that $Ric \ge -(n-1)b^2 $ for some $b >0$, then 
$$ \frac{n-1}{r} + \frac{\partial}{\partial r} \big(\ln( A_x(r,\theta))\big) \leq  (n-1)b\, \coth (br),$$
and
$$r^{n-1}A_x(r,\theta)\le \left(\frac{\sinh (br)}{b}\right)^{n-1}. $$

If $(M,g)$ now satisfies $K_{g} \leq -a^2\leq 0$, then 
$$ \frac{n-1}{r} + \frac{\partial}{\partial r} \big(\ln( A_x(r,\theta))\big) \geq  (n-1)a\, \coth (ar).$$
\end{thm}
Next, we recall the Hessian comparison theorem: 
\begin{thm}\label{Hessian} Let $(M,g)$ be a Riemannian manifold such that  $K_{g} \leq -a^2$ with $a>0$. Let $y\in M$, then at any point $x \neq y$, it holds 
$$\hes (d_{g}(y,\cdot) )(x)\geq  {a} \coth (a \,d_g(y,x))\bar{g}, $$
where $\hes$ denotes the Hessian of the distance function and $\bar{g}$ the restriction of the metric $g$ to $\{\nabla_g d_g(y,\cdot)(x)\}^{\perp} \subset T_xM$. Taking the trace, we get
$$ \Delta_{g} (d_g(y,\cdot)) (x) \geq (n-1)  {a}\, \coth (a \,d_g(y,x)).$$

If $a=0$ then
$$ \hes (d_{g}(y,\cdot) )(x)\geq \frac{1}{d_g(y,x)}\bar{g} \mbox{ and } \Delta_{g} (d_g(y,\cdot)) (x) \geq \frac{(n-1)}{d_g(y,x)}.$$
\end{thm}

%
Finally, we recall the Laplacian comparison theorem:
\begin{thm}\label{Laplace} Let $(M,g)$ be a Riemannian manifold such that  $Ric_{g} \geq -(n-1) b^2$. Let $y\in M$, then at any point $x \neq y$, it holds 
$$ \Delta_{g} (d_g(y,\cdot))(x) \leq (n-1)  {b}\, \coth (b \,d_g(x,y)).$$
\end{thm}

\subsection{Poincar\'e type Inequalities} In this final subsection we recall some inequalities in Sobolev space and deduce some corollaries.\\
The following theorem is due to McKean for $p=2$ (see \cite{chavel}) and generalized further by Strichartz \cite[Theorem 5.4]{Stri}.

\begin{thm}\label{mckean} Let $(M,g)$ be a Hadamard manifold with $ K_g \le -a^2<0$ then for $1\le p < \infty $ the inequality
\begin{equation}\label{mc-st}
\left(\frac{(n-1)a}{p}\right)^p \int\limits_M |u|^p \ d\mu_g \ \le  \ \int\limits_M|\nabla_g u|^p \ d\mu_g
\end{equation}
holds for all $u\in C_c^\infty(M).$  
\end{thm}
\begin{thm} Let $(M,g)$ be a complete Riemannian manifold satisfying $Ric_g \ge -(n-1)b^2 $, and whose spectral gap is positive. Then, there exists a constant $C_p>0$ such that
\begin{equation}\label{multiplicative}
|| |\nabla_gu| ||_p \ \le \ C_p \, ||u||_p^\frac{1}{2} \, || \Delta_g u ||_p^\frac{1}{2}  
\end{equation}
holds for all $u\in C_c^\infty(M).$ 
\end{thm}

\begin{proof}
The result follows from two main ingredients. First, we use

$$|| \Delta_g ^{1/2} u||_p \leq C_p \, ||u||_p^\frac{1}{2}  \,|| \Delta_g u ||_p^\frac{1}{2},  $$
where $\Delta_g ^{1/2} u$ stands for the fractional Laplacian. This inequality holds on any complete Riemannian manifolds \cite{Komatsu} (see also \cite[Proposition 2.2]{TCOUL}).

Then, we get the result by combining this together with the boundedness of the Riesz transform:

$$|| |\nabla_gu| ||_p \ \le \ C_p  \,|| \Delta_g ^{1/2} u||_p,$$
which holds under these assumptions as proved in \cite[Theorem 1.9]{Auscher}; see also \cite[Theorem 4.1]{Bakry_Riesz}.

\end{proof}
Combining the above two theorems and a recursive application will give the following inequality:\\
\begin{thm}\label{mc-st-k} Let $(M,g)$ be a Hadamard manifold with $Ric_g \ge -(n-1)b^2 $ and $ K_g \le -a^2<0$, then for $1<p <+ \infty $ and positive  $k\in \mathbb{N}$, there exists $C_{k,p}>0$  such that
\begin{equation}
 \int\limits_M |u|^p \ d\mu_g \ \le  \ C_{k,p}\int\limits_M|\nabla_g^k u|^p \ d\mu_g
\end{equation}
holds for all $u\in C_c^\infty(M).$  
\end{thm}

\section{Green's function} One of the crucial tools which we will be using to prove our results is the information on the Green function of the Laplace operator. In this section, following the approach due to Li and Tam \cite{LT}, we will construct a Green function on a Hadamard manifold and show that it can be bounded by terms depending only on the curvature bounds; we will also establish sharp integral estimates for this Green function and its gradient. First, let us recall the definition of entire Green's function. 
\subsection{Green's Function: Definition and Model cases.} In this subsection we define the notion of entire Green's function and recall the Green  function of the model cases.
\begin{defi}
Let $(M,g)$ be a Riemannian manifold, then an entire Green's function of the Laplace Beltrami operator $-\Delta_g$ is a function $G:M\times M \setminus \{(x,x) : x \in M \} \rightarrow [0,\infty) $ satisfying 
\begin{itemize}
\item[(i)] For each fixed $x\in M,\;\Delta_g G^x(y)=0$ for all $y\in M\setminus\{x\}$, where $G^x$ is the function $y\rightarrow G(x,y)$.
\item[(ii)] $G(x,y)= G(y,x)$ for all $x\not=y$.
\item[(iii)]For each fixed $x\in M$,
$$G^x(y)\; =\; \begin{cases} \frac{d_g(x,y)^{2-n}}{(n-2)\omega_{n-1}}\left[1+o(1)\right] \ \ \ \ \ \ \ \ \ {\text if} \ \ n \ge3, \\
\frac{-\log d_g(x,y)}{2\pi} \left[1+o(1)\right] \ \ \ \ \ \ \ \ {\text if} \ \ n=2.
\end{cases}
 $$    
\end{itemize}

\end{defi}
Let $\Phi:(0,\infty)\rightarrow (0,\infty)$ be defined by 
\begin{equation}\label{Euc-funda}
\Phi(r)=\frac{r^{2-n}}{(n-2)\omega_{n-1}},
\end{equation}
then we know that an entire Green's function of $-\Delta$ in the Euclidean space $\R^n,n\ge 3$, is given by $G(x,y)=\Phi(|x-y|)$. Similarly, for $a>0$, if $\Psi_a:(0,\infty)\rightarrow (0,\infty)$  is defined by 
\begin{equation}\label{Hyp-funda}
\Psi_a(r)= \frac{a^{n-2}}{\omega_{n-1}}\int\limits_{ar}^\infty (\sinh t)^{1-n}\;dt,
\end{equation}
 then one can easily see that an entire Green's function of the hyperbolic space $\mathbb{H}^n_{-a^2}$ is given by $ G(x,y)\;=\; \Psi_a(d_{g_a}(x,y))$, where $d_{g_a}$ is the Riemannian distance in $\mathbb{H}^n_{-a^2}$.
 \subsection{Green's function: Existence and Estimates. } In the case of general Hadamard manifolds, we have the following theorem which establishes the existence of entire Green's function:
\begin{thm}\label{greens} Let $(M,g)$ be a Hadamard manifold of dimension $n\ge 3$, then $(M,g)$ admits an entire Green's function $G$ satisfying the estimate
\begin{equation}\label{eucli-estimate}
0\; < \; G(x,y) \le  \Phi(d_g(x,y)),
\end{equation}
where $\Phi$ is as in \eqref{Euc-funda}. Moreover, if $(M,g)$ satisfies:\\ 
(i)$\;\;K_{g} \le -a^2<0$, then  
\begin{equation}\label{hyp-estimate}
 0\; < \; G(x,y) \le \Psi_a(d_g(x,y)),
\end{equation}
(ii)$\;\; Ric_g \geq -(n-1)b^2\;,b>0$, then 
\begin{equation}\label{ricci-estimate}
0< \Psi_b(d_g(x,y)) \le G(x,y),
\end{equation}
where $\Psi_a$ and $\Psi_b$ are as in \eqref{Hyp-funda}.
\end{thm}

We need the entire Green function for the following representation formula:

\begin{rem} We will observe from the proof that the Green function $G$ established in the previous theorem satisfies for every $u\in C^2_c(M)$ 
\begin{equation}\label{repre-1}
u(x)\;=\;\int\limits_M G(x,y)(-\Delta_gu(y))\;d\mu_g (y)  
\end{equation}
and 
\begin{equation}\label{repre-2}
u(x)\;=\; -\Delta_g\left(\int\limits_M G(x,y)u(y)\;d\mu_g (y)\right) 
\end{equation}

\end{rem}
The next theorem gives us precise asymptotic bounds of $G$ and its gradient near the singularity. These bounds will be crucial to prove Adams inequalities for the best exponents.
\begin{thm} \label{greens-sing-estimate} Let $(M,g)$ be a Hadamard manifold satisfying $Ric_g \geq -(n-1)b^2$ for some $b>0 $. Let $G$ be the entire Green function established in Theorem \ref{greens}, then for every $R>0$ there exist positive constants $A, B$ depending only on $R$ such that
\begin{equation}\label{upper-lower-bound}
\Phi(d_g(x,\cdot)) \left[1-A\, [d_g(x,\cdot)]^2\right] \le G(x,\cdot) \le  \Phi(d_g(x,\cdot)),
\end{equation}
and
\begin{equation}\label{gradient-singularity}
|\nabla_g G(x,\cdot) \ - \nabla_g \Phi(d_g(x,\cdot))| \le  \left\{\begin{array}{ll}
B[d_g(x,\cdot)]^{3-n} & {\rm if} \ n>3,\\
B(1+|\log d_g(x,\cdot)|) & {\rm if} \ n=3,\\
\end{array}\right.
\end{equation}
holds in $B(x,R)$, uniformly for all $x \in M$ .
 \end{thm}
In addition to the above pointwise estimates, we also need  estimates on the $L^2$ and $L^1$ norms of $G$ and its gradient:
\begin{thm}\label{greens-estimate} Let $(M,g)$ be a Hadamard manifold satisfying $ K_g \leq -a^2$ and $Ric_g \geq -(n-1)b^2$ for some $a>0,b>0 $. Let $G$ be the entire Green function established in Theorem \ref{greens}, then  there exists a $C>0$ such that for every $x\in M$ and every $R>0$,
\begin{equation}\label{local-estimate}
\int\limits_{B(x,R)} G(x,\cdot)\;d\mu_g \le C(1+R),
\end{equation}
\begin{equation}\label{local-gradient-estimate}
\int\limits_{B(x,R)} |\nabla_gG(x,\cdot)|\;d\mu_g \le C(1+R),
\end{equation}
\begin{equation}\label{infinity-estimate}
\int\limits_{M\setminus B(x,R)} G^2(x,\cdot)\;d\mu_g \le C \Psi_a(R),
\end{equation}
and
\begin{equation}\label{gradient-estimate}
\int\limits_{M\setminus B(x,R)} |\nabla_g G(x,\cdot)|^2\;d\mu_g \le \frac{C}{R^2} \Psi_a(R).
\end{equation}
\end{thm}
\medskip

\subsection{Proofs of Theorems.} We need a few lemmas before going into the proofs of Theorem \ref{greens}, Theorem \ref{greens-sing-estimate}, and Theorem \ref{greens-estimate}. First,
let us recall the theorem concerning the existence of Green's function for the Laplace operator with Dirichlet boundary condition in bounded domains. For details we refer to \cite{Aubin}.
\begin{lem}\label{dirichlet-greens} Let $(M,g)$ be a Riemannian manifold of dimension $n\ge 3$, and $\Omega$ be a bounded open subset of $M$ with smooth boundary, then there exists  $G : \overline{\Omega}\times\overline{\Omega} \setminus \{(x,x): x\in \overline{\Omega}\}\rightarrow [0,\infty)$ such that
\begin{enumerate}
\item[(i)] $\;\;G(x,y)=G(y,x)\; ,\; \forall x\not= y, $
\item[(ii)] $\;\;G(x,y)=0 \; ,\; {\rm if}\; x\in \partial\Omega $ or $y\in \partial\Omega,$
\item[(iii)]$\;\;-\Delta_g G(x,\cdot) = \delta_x, \; -\Delta_g G(\cdot,y) = \delta_y $,
\item[(iv)]For each $x$ fixed, $\;\;G(x,y) = \frac{[d_g(x,y)]^{2-n}}{(n-2)\omega_{n-1}}\left[1+o(1)\right] $ as $y\rightarrow x.$
\end{enumerate}
\end{lem}
We are going to get our Green function as the limit of Dirichlet Green's functions in bounded domains. The following lemma plays a crucial role in getting the bounds on the Green function.
\begin{lem}\label{comp-euclidean} Let $(M,g)$ be a Hadamard manifold of dimension $n\ge 3$, and $\Phi,\Psi_a$ be as in \eqref{Euc-funda} and \eqref{Hyp-funda}.  For $x\in M$, define $\Phi^x,\Psi^x_a : M\setminus \{x\}\rightarrow (0,\infty)$ by 
$ \Phi^x (y)\;=\; \Phi(d_g(x,y))\;$ and $\Psi^x_a (y)\;=\: \Psi_a(d_g(x,y))\;$.
Then:\begin{enumerate}
\item[(i)]$-\Delta_g\Phi^x(y) \ge 0 \;\; {\rm for \;all\;} y\in M \setminus\{x\}. $
\item[(ii)]$-\Delta_g\Psi_a^x(y) \ge 0 \;\; {\rm for \;all\;} y\in M \setminus\{x\} $ if $K_{g} \le -a^2<0$.
\item[(iii)]$-\Delta_g\Psi_b^x(y) \le 0 \;\; {\rm for \;all\;} y\in M \setminus\{x\} $ if $Ric_{g} \geq -(n-1)b^2$.
\end{enumerate}
\end{lem}
\begin{proof} 
Let us recall that, given a $C^2$ function $f: (0,+\infty) \rightarrow (0,+\infty)$, $$ \Delta_g( f(d_g(x,\cdot)) = f'(d_g(x,\cdot)) \Delta_g d_g(x,\cdot) + f''(d_g(x,\cdot))$$ (we use $|\nabla_g d_g(x,\cdot)|=1$ on $M \setminus \{x\}$). Note that $\Phi''(r) = \frac{n-1}{r} \Phi'(r)$; a similar formula holds for $\Psi_a$. The conclusions (i) and (ii) then follow from Theorem \ref{Hessian}  while (iii) follows from Theorem \ref{Laplace}.
\end{proof}
\noindent
{\bf Proof of Theorem \ref{greens}.} For $x \in M$ and $R>0$, we denote by $B(x,R)$ the open Riemannian ball of radius $R$ centered at $x$. Fix a point $O \in M$ and define for $R>0$, $B_R := B(O,R)$.  
Let $G_R$ denote the unique Dirichlet Green function of $B_R$ given by Lemma \ref{dirichlet-greens}; we will show that the limit of $G_R$ as $R\rightarrow \infty$ exists and is the required Green function. We will present the arguments in several steps.\\\\
{\it Step 1:} Let $0<R_1<R_2 <\infty$ and $x \neq y\in B_{R_1}$, then $G_{R_1}(x,y)\le G_{R_2}(x,y).$\\
{\it Proof of Step 1.} Fix $x \in B_{R_1}$, $\epsilon >0$, and consider the function $g_\epsilon : B_{R_1}\setminus \{x\} \rightarrow \R $ defined by
$$g_\epsilon(y) = (1+\epsilon)G_{R_2}(x,y) - G_{R_1}(x,y). $$
Then for any small $\delta >0$, $g_\epsilon$ is harmonic in $B_{R_1}\setminus \overline{B(x,\delta)}$, and $g_\epsilon \ge 0$ on $\partial(B_{R_1}\setminus \overline{B(x,\delta)} )$ thanks to (iv) of Lemma \ref{dirichlet-greens}.
Thus, by maximum principle $g_\epsilon \ge 0$ in $B_{R_1}\setminus \overline{B_{\delta}} $ for $\delta$ small enough, and hence in $B_{R_1}\setminus \{x\}.$ Now, Step 1 follows by taking $\epsilon \rightarrow 0.$\\\\
{\it Step 2:}  For every $R>0$, $G_{R}(x,y) \le \Phi^x(y)$ for all $x,y\in B_R,$ where $\Phi^x$ is defined as in Lemma \ref{comp-euclidean}.\\
{\it Proof of Step 2.} Fix $x\in B_R$ and $\delta >0$ small enough, and consider the function $g^{x,\delta} : B_R\setminus B(x,\delta)\rightarrow \R$ defined by
$$g^{x,\delta}(y) = \Phi^x(y) - m_\delta G_{R}(x,y), $$ 
where $m_\delta = \frac{\Phi(\delta)}{\max\{G_R(x,y): d_g(x,y)=\delta\}}$.
Then, it follows from the maximum principle that $g^{x,\delta}(y)\ge 0$ in $B_R\setminus B(x,\delta)$ . Note that $m_\delta \rightarrow 1$ as $\delta \rightarrow 0.$ Thus, Step 2 follows by taking $\delta \rightarrow 0$ in $g^{x,\delta}(y)\ge 0$ for $y\in B_R\setminus B(x,\delta)$.\\\\
{\it Step 3:} Define for $x,y\in M,\; x\not=y, G(x,y)= \lim\limits_{R\rightarrow \infty}G_R(x,y)$, then $G$ is the required Green function.\\
 {\it Proof of Step 3.} First, observe that $G$ is well-defined thanks to Step 1 and Step 2. The estimate \eqref{eucli-estimate} on $G$ follows from Step 2 by taking the limit $R \rightarrow \infty.$ Also, $G(x,y)=G(y,x)$ as it holds for each $G_R.$ 
 For any $x\in M$, the function $\Phi_x \in L^1_{loc}(M)$ and $G_R \le \Phi_x.$
 Thus $\Delta_gG_R(x,.)\rightarrow \Delta_gG(x,.)$ in the sense of distributions, which implies $-\Delta_gG(x,.) = \delta_x$ for all $x\in M$, in particular $\Delta_gG^x =0$ in $M\setminus \{x\}.$ \\ It remains to show that $G$ satisfies the last condition of the definition of entire Green's function.
 Fix $x\in M$ and $R>0$ such that $x\in B_{\frac{R}{2}}$, then as $y\rightarrow x$, we have
$$ \frac{[d_g(x,y)]^{2-n}}{(n-2)\omega_{n-1}}\left[1+o(1)\right]= G_R(x,y)\le G(x,y) \le \frac{[d_g(x,y)]^{2-n}}{(n-2)\omega_{n-1}}$$
and hence $G$ satisfies (iii) of the definition.\\\\
When $(M,g)$ satisfies $K_g \le -a^2<0$, we can repeat Steps 2 and 3 with $\Psi_a$ instead of $\Phi$ to establish \eqref{hyp-estimate}.\\\\
To prove \eqref{ricci-estimate}, fix $x\in M$. For $\delta>0$ define $h^{x,\delta}$ by
$$h^{x,\delta}(y) \;=\; m_\delta G^x(y) - \Psi_b(d_g(x,y))\;,\; y\in M\setminus B(x,\delta),$$
where $m_\delta = \frac{\Psi_b(\delta)}{\min\limits_{\{y: d_g(x,y)=\delta\}}G^x(y)}$. Then, using (iii) of Lemma \ref{comp-euclidean} we get $-\Delta_g h^{x,\delta} \ge 0$, and hence using the maximum principle $h^{x,\delta}\ge 0$ in $ M\setminus B(x,\delta)$. Taking the limit as $\delta \rightarrow 0$, and observing that $m_{\delta}\rightarrow 1$, we get $G^x(y) - \Psi_b(d_g(x,y))\ge 0$ for $y\in M\setminus \{x\}.$ This completes the proof of the theorem.
\qed
\\\\
{\bf Proof of Theorem \ref{greens-sing-estimate}.} The upper and lower bounds of $G$, namely \eqref{upper-lower-bound}, follow from \eqref{eucli-estimate} and \eqref{ricci-estimate}.\\\\
To prove the estimate on the gradient, first note that we have the following   pointwise estimate which follows from \cite{Yau} and the subsequent improvement obtained in \cite{LW}: There exists positive constants $C_1,C_2$ depending on the lower Ricci curvature bound and the dimension $n$ such that
 \begin{equation}\label{gradient-point-est}
|\nabla_g(\log G(x,\cdot))| \;\le \; \frac{C_1}{d_g(x,\cdot)} + C_2.
\end{equation}

Combining this with the estimate on $G$, we get the existence of a positive constant $C$ such that on $B(x,R) $, uniformly in $ x\in M$,
\begin{equation}\label{gradient-point-est-1}
|\nabla_gG(x,\cdot)| \;\le \; \ C \left[d_g(x,\cdot)\right]^{1-n}.
\end{equation} 
Let $\Phi$ be as in \eqref{Euc-funda}, then using the notation in Theorem \ref{volumeelement}, we get 
$$ -\Delta_g \Phi ({\rm exp}_x(r\, \theta))= \Phi^\prime(r)\frac{\partial}{\partial r}\ln  (A_x(r,\theta))\ \ {\rm in}\ \ M\setminus \{x\},$$
where $\frac{\partial}{\partial r}$ denotes the radial derivative in normal coordinates centered at $x$, and $\Exp_x$ stands for the Riemannian exponential map at $x$. Using our curvature bound, we infer from Theorem \ref{volumeelement}, the estimate $|\frac{\partial}{\partial r}\ln (A_x(r,\theta))| \le Cr$ where $C$ is uniform in $x$ and $r \leq R$. Thus, the function $H(x,\cdot)$ defined by $H(x,{\rm exp}_x(r\,\theta)) \ := \ \Phi^\prime(r)\frac{\partial}{\partial r}\ln (A_x(r,\theta)) $ satisfies the following estimate on $B(x,R) $, uniformly in $ x\in M$,
\begin{equation}\label{h-estimate}
|H(x,\cdot) |\ \le C \left[d_g(x,\cdot)\right]^{2-n}.
\end{equation}
We also have in the sense of distributions 
\begin{equation}
-\Delta_g \Phi (d_g(x,\cdot))= \delta_x \ + \ H(x,\cdot)\ \ {\rm in}\ \ M,
\end{equation}
where $\delta_x$ denotes the Dirac delta distribution at $x.$ \\
Fix $R>0$ and choose a smooth function $f :[0,\infty)\rightarrow \R$ such that $f=1$ on $[0,R]$ and $f=0$ in $[2R,\infty)$. For $x\not= y\in M$, define $U(x,y)$ by 
\begin{equation}
U(x,y) \ := \ \int\limits_MG(y,z)f(d_g(x,z))H(x,z)\ d\mu_g(z).
\end{equation}
Using the estimates on $G$ and $H$, we can see that in the sense of distributions
\begin{equation}
-\Delta_g U(x,\cdot) \ = \ f(d_g(x,\cdot))H(x,\cdot), 
\end{equation}
and hence as distributions 
\begin{equation}
-\Delta_g \left[G(x,\cdot)\ - \ \Phi (d_g(x,\cdot))\ +\ U(x,\cdot)\right] \ =  0 \ \ {\rm in}\ \ B(x,R).
\end{equation}
In other words, the function $h^x$ defined by
$$h^x(y)= G(x,y)\ - \ \Phi (d_g(x,y))\ +\ U(x,y)$$
is harmonic in $B(x,R)$, and we claim $h^x$ is bounded on $\partial B(x,R)$, uniformly in $x$. This claim follows once we prove the same property for $U(x,\cdot)$. We will estimate $U(x,\cdot)$  by writing it in the normal coordinates centered at $x.$  Let us identify isometrically  the tangent space of $M$ at $x$ with the Euclidean space $\R^n$ by fixing a $g$-orthonormal basis.  Since $ K_g \le 0$, by Rauch's comparison theorem, we get for any two points $z_i\in \R^{n}, i=1,2$, 
$$d_g(\Exp_x(z_1),\Exp_x(z_2)) \ge |z_1-z_2| .$$ 
Since $\Phi$ is decreasing, we get $\Phi(d_g(\Exp_x(z_1),\Exp_x(z_2))) \le \Phi(|z_1-z_2|)$. We also set $\Exp_x^{-1}(z)= \tilde{z}$ for an arbitrary point $z \in M$.

Using the lower Ricci curvature bound, we can estimate from above the volume element; precisely, if we set $d \tilde{z}$ the Lebesgue measure,  Theorem \ref{volumeelement} can be rephrased as 
$$d\mu_g (z) \leq  \left(\frac{\sinh (b|\tilde{z}|)}{b|\tilde{z}|}\right)^{n-1} \, d\tilde{z}.$$
Thus, 
$$|U(x,y)| \le C\int\limits_{B(x,2R)}\Phi(d_g(y,z))\left[d_g(x,z)\right]^{2-n} d\mu_g(z)$$
$$ \le C\int\limits_{|\tilde{z}|<2R}|\tilde{y}-\tilde{z}|^{2-n}|\tilde{z}|^{2-n} \left(\frac{\sinh (b|\tilde{z}|)}{b|\tilde{z}|}\right)^{n-1} \, d\tilde{z}$$
$$\le C\int\limits_{|\tilde{z}|<2R}|\tilde{y}-\tilde{z}|^{2-n}|\tilde{z}|^{2-n}  d\tilde{z}\le \left\{\begin{array}{ll}
C|\tilde{y}|^{4-n} & {\rm if} \ n>4,\\
C(1+|\ln |\tilde{y}||) & {\rm if} \ n=4,\\
C & {\rm if} \ n=3,\\
\end{array}\right.$$
where $C$ depends on $R$ via $\max\limits _{0<t<2R}(\frac{\sinh bt}{bt})^{n-1}$. Going back to the original variables we get 
$$ |U(x,y)| \le \left\{\begin{array}{ll}
C \,d_g(x,y)^{4-n} & {\rm if} \ n>4,\\
C(1+|\ln d_g(x,y)|) & {\rm if} \ n=4,\\
C & {\rm if} \ n=3.\\
\end{array}\right.$$
This proves the uniform bound of $U$ and hence $h^x$ on $\partial B(x,R)$. Since $h^x$ is harmonic in $B(x,R)$, the gradient of $h^x$ is uniformly bounded in $ B(x,\frac{R}{2})$ thanks to the gradient estimate already mentioned in \eqref{gradient-point-est}. 
Thus, 
$$\nabla_g G(x,\cdot) = \nabla_g \Phi(d_g(x,\cdot)) - \nabla_gU(x,\cdot) + \nabla_g h^x,$$
and hence it remains to estimate $\nabla_gU(x,\cdot) $.\\
By definition of $U,$
 $$ \nabla_gU(x,\cdot) = \ \int\limits_M\nabla_gG(y,\cdot)f(d_g(x,\cdot))H(x,\cdot)\ d\mu_g,$$
thus, using the estimates \eqref{gradient-point-est-1} and \eqref{h-estimate}, and proceeding exactly as we estimated $U$ above,  we get, when $n>3$, the estimate $ |\nabla_gU(x,\cdot) | \le C [d_g(x,\cdot)]^{3-n}$, and,  when $n=3$, \ $ |\nabla_gU(x,\cdot) | \le C [1+ |\log d_g(x,\cdot)|]$. This completes the proof.
\qed
\\\\
{\bf Proof of Theorem \ref{greens-estimate}.}  Fix $x\in M$ and recall that  $G^x(y)=G(x,y)$. Then, it follows from Theorem \ref{greens} that
\begin{equation}\label{comp-ball}
B(x,\Psi_b^{-1}(t)) \subset \{y : G^x(y)>t\} \subset B(x,\Psi_a^{-1}(t)).
\end{equation}
Let $B_R$ and $G_R$ be as in the proof of Theorem \ref{greens}. Define for $y\not=x$, $G_R^x(y):= G_R(x,y)$, then we know that $G_R^x$ monotonically converges to $G^x$. For $t>0$ and $R>0$, define the compactly supported function $$H_R^t(y) \;=\; \min\{t,G_R^x(y)\}.$$
Using Theorem \ref{mckean} with $p=2$, we get 
\begin{equation}\label{poincare-estimate}
\left[\frac{(n-1)a}{2}\right]^2\int\limits_{B_R}(H_R^t)^2\;d\mu_g\;\le\; \int\limits_{B_R}|\nabla_g H_R^t|^2 d\mu_g.
\end{equation}
Now, 
\begin{eqnarray}\label{eqHR}
\int\limits_{B_R}|\nabla_g H_R^t|^2 d\mu_g  & =& \int\limits_{B_R\cap \{G^x_R<t\}}|\nabla_g G_R^x|^2 d\mu_g \nonumber \\
&=& -\int\limits_{B_R\cap \{G^x_R<t\}}(\Delta_g G_R^x)G_R^x d\mu_g \;-\int\limits_{ \{G^x_R=t\}}\left(\frac{\partial G_R^x}{\partial \nu}\right)G_R^x \nonumber \\
&=& -t\, \int\limits_{ \{G^x_R=t\}}\frac{\partial G_R^x}{\partial \nu},
\end{eqnarray}

where $\nu$ is the outward unit normal of $\{G^x_R  > t\}$, and we have used  $\Delta_gG^x_R=0$ in $M\setminus \{x\}$. For small enough $\epsilon >0$, we get by applying Green's formula on $\{G^x_R  > t\}\setminus B(x,\epsilon)$: 
$$ \int\limits_{ \{G^x_R=t\}}\frac{\partial G_R^x}{\partial \nu} +  \int\limits_{\partial B(x,\epsilon) } \frac{\partial G_R^x}{\partial \nu} = \int\limits_{ \{G^x_R >t\}\setminus B(x,\epsilon)}  \Delta_g G_R^x\  d\mu_g  =0,$$

where $\nu$ on $\partial B(x,\epsilon)$ is the unit inward normal of $B(x,\epsilon)$.
 Inserting this relation into \eqref{eqHR},  we get, by definition of $G_R^x$,
$$\int\limits_{B_R}|\nabla_g H_R^t|^2 d\mu_g = t\, \lim\limits_{\epsilon \rightarrow 0}\int\limits_{ \partial B(x,\epsilon)}(\frac{\partial G_R^x}{\partial \nu}) =t.$$
Using this estimate in \eqref{poincare-estimate}, and taking the limit $R\rightarrow \infty$ we get
\begin{equation}\label{basic-estimate}
\left[\frac{(n-1)a}{2}\right]^2\left[t^2 \mu_g(\{G^x\geq t\}) + \int\limits_{\{G^x<t\}}(G^x)^2\;d\mu_g\right]\;\le\; t.
\end{equation}
Hence $ \int\limits_{\{G^x<t\}}(G^x)^2\;d\mu_g\;\le Ct$ and \eqref{infinity-estimate} follows from \eqref{comp-ball}.\\To prove \eqref{local-estimate}, first observe from \eqref{basic-estimate} that 
\begin{equation}\label{sublevel-measure}
\mu_g(\{G^x>t\}) \le \left[\frac{2}{(n-1)a}\right]^2\frac{1}{t}, \;\; \forall \;t>0.
\end{equation}
Also from \eqref{comp-ball}, Theorem \ref{volume}, and \eqref{volume-model} we have,
\begin{equation}\label{sublevel-infinity}
\mu_g(\{G^x>t\}) \le  \mu_g(B(x, \Phi^{-1}(t)))\le V^n_{-b^2}(\Phi^{-1}(t))\le C\left(\frac{1}{t}\right)^{\frac{n}{n-2}} ,\; t \ge 1.
\end{equation}
Thus, using \eqref{comp-ball} and \eqref{sublevel-measure}, we get
\begin{eqnarray*}\int\limits_{B(x,R)} G(x,y)\;d\mu_g(y) & \leq  & \int\limits_{G^x>\Psi_b(R)} G^x(y)\;d\mu_g(y) \\ &=& \int\limits_0^\infty\mu_g\left(\{G^x>t\}\cap \{G^x>\Psi_b(R)\}\right)dt \\
&=& \int\limits_0^{\Psi_b(R)}\mu_g\left(\{G^x>\Psi_b(R)\}\right)dt+\int\limits_{\Psi_b(R)}^\infty\mu_g\left(\{G^x>t\}\right)dt \\
&=& \left[\frac{2}{(n-1)a}\right]^2 + \int\limits_{\Psi_b(R)}^\infty\mu_g\left(\{G^x>t\}\right)dt.
\end{eqnarray*}
If $\Psi_b(R)\ge 1$, then \eqref{sublevel-infinity} implies that $\int\limits_{\Psi_b(R)}^\infty\mu_g\left(\{G^x>t\}\right)dt \le C$, where $C$ is independent of $x$. If $\Psi_b(R)< 1$, then \eqref{sublevel-measure} and  \eqref{sublevel-infinity} give
$$\int\limits_{\Psi_b(R)}^\infty\mu_g\left(\{G^x>t\}\right)dt= \int\limits_{\Psi_b(R)}^1\mu_g\left(\{G^x>t\}\right)dt+ \int\limits_1^\infty\mu_g\left(\{G^x>t\}\right)dt \le C(1+R).$$
This proves \eqref{local-estimate}. To prove \eqref{local-gradient-estimate}, first observe that if $R\le 1$, then using \eqref{gradient-point-est-1} we get
$$\int\limits_{B(x,R)}|\nabla_gG(x,\cdot)|d\mu_g \le C \int\limits_{B(x,R)}[d_g(x,\cdot)]^{1-n}d\mu_g \le C$$
uniformly in $x$ thanks to Theorem \ref{volumeelement}. This together with the estimate $$|\nabla_gG(x,\cdot)| \le C G(x,\cdot) \ \ {\rm in}\ \  M\setminus B(x,1)$$ 
(which follows from \eqref{gradient-point-est}), and \eqref{local-estimate} prove \eqref{local-gradient-estimate}.\\

The last identity \eqref{gradient-estimate} follows from \eqref{infinity-estimate} once we use the estimate \eqref{gradient-point-est}. We can also have the following alternate proof: \\
Choose a smooth function $f: \R \rightarrow [0,1]$ such that $f(r)=0$ if $r\le 1$ and $f(r)=1$ if $r\geq 2$, and define $f_R: M \rightarrow [0,1]$ by $f_R(y)= f(\frac{d_g(x,y)}{R})$.\\  Since $\Delta_gG^x_{\tilde R} =0$ in $ B_{\tilde{R}}\setminus \{x\}$, we get
$$\int\limits_{B_{\tilde{R}}}\Delta_gG^x_{\tilde{R}}(y) (f_R(y))^2G_{\tilde R}^x(y)\;d\mu_g(y)=0.$$
This implies
$$
\int\limits_{B_{\tilde{R}}}|\nabla_gG^x_{\tilde{R}}(y)|^2 (f_R(y))^2\;d\mu_g(y) \le 2\int\limits_{B_{\tilde{R}}}|\nabla_gf_R(y)| |\nabla_gG^x_{\tilde{R}}(y)| f_R(y)G^x_{\tilde{R}}(y)\;d\mu_g(y) $$
$$\le \frac{C}{R}\left(\int\limits_{B_{\tilde{R}}}|\nabla_gG^x_{\tilde{R}}(y)|^2 (f_R(y))^2\;d\mu_g(y) \right)^\frac{1}{2}\left(\int\limits_{\{y:R\le d_g(x,y)\le 2R\}}(G^x_{\tilde{R}})^2\right)^{^\frac{1}{2}}.$$
Now, \eqref{gradient-estimate} follows by taking ${\tilde{R}}\rightarrow \infty$ and using  \eqref{infinity-estimate}. 

\qed
\section{Proof of Theorem} In this section we will prove our main theorem. We follow the idea of converting the problem into a convolution type estimate problem introduced by Adams \cite{A} and further developed by Fontana \cite{Fonta} and Fontana-Morpurgo \cite{FonMo,FontaMo}. First, we will introduce these kernels and prove the necessary estimates on them using the estimates on $G$ and its gradient established in Section 3.
\subsection{Estimates on the Kernel.}

For positive $m \in \mathbb{N}$, we define the kernel \\$$K^m : M\times M \setminus \{(x,x) : x \in M \} \rightarrow (0,\infty)$$  by 
\begin{align}\label{kernel-defi}
K^m(x,y)=\begin{cases}
|\nabla_g G(x,y)|\ \ \ \ \ \ \ \ \ \   \mbox{if} \ \ \ m=1,\\
G(x,y)\ \ \ \ \ \ \ \ \ \ \ \ \ \ \ \  \mbox{if} \ \ \ m=2,\\
\int\limits_MK^{m-2}(x,\cdot)G(\cdot,y)\;d\mu_g\ \ \ \ \ \ \ \mbox{if}\ m \  \mbox{is even},\\
\int\limits_MK^{m-1}(x,\cdot)|\nabla_gG(\cdot,y)|\;d\mu_g\ \ \ \mbox{if}\ m \  \mbox{is odd}.
\end{cases}
\end{align}
First, we will show that $K^m$ is well-defined and satisfies the required estimates.
\begin{lem}\label{kernel-estimate} Let $(M,g)$ be an $n$-dimensional Hadamard manifold satisfying $ K_g\le -a^2$ and $Ric_g\ge -(n-1)b^2$ for some positive numbers $a,b$, then for $m<n$, $K^m$ is well-defined and satisfies the estimate

\begin{align}\label{ker-growth}
K^m(x,y)\leq \begin{cases}
\alpha_{n,m}\left[d_g(x,y)\right]^{m-n}\left(1+ C\left[d_g(x,y)\right]^{\frac{1}{2}}\right)\ \ \ \mbox{if} \ \ \ d_g(x,y)<1,\\
C e^{-\beta_m d_g(x,y)}\ \ \ \mbox{if} \ \ \ \ \ \ \ \ \ \ d_g(x,y) \geq 1,\\
\end{cases}
\end{align}
for some $\beta_m >0$, $C>0$, and $\alpha_{n,m}$ is given by
\begin{align}
\alpha_{n,m} = \begin{cases} \frac{\Gamma(\frac{n-m}{2})}{\omega_{n-1}2^{m-1}(\frac{m-2}{2})!\Gamma(\frac{n}{2})} \ \ \mbox{if}\ m \  \mbox{is even},\\
\frac{\Gamma(\frac{n-m+1}{2})}{\omega_{n-1}2^{m-1}(\frac{m-1}{2})!\Gamma(\frac{n}{2})} \ \ \mbox{if}\ m \  \mbox{is odd}.
\end{cases}
\end{align}
Moreover, there exists $\alpha_m>0$ and $C>0$ such that
\begin{equation}\label{infinity-kernel}
\int\limits_{M\setminus B(x,R)} \left(K^{m}(x,y)\right)^2\;d\mu_g(y) \le C \e^{-\alpha_m R},\;\; {\rm for \; all}\; R\ge 1. 
\end{equation}
\end{lem}
\begin{proof}  First, observe that when $m=1$ the lemma follows from \eqref{gradient-singularity}, \eqref{gradient-point-est}, and \eqref{gradient-estimate}. When $m=2$, it again follows from \eqref{hyp-estimate} and the estimate \eqref{infinity-estimate}. Next, we show that if the lemma is true for an even $m$ then it holds for $m+i$ with $i\in \{1,2\}$ provided $m+i<n$, and hence it will follow for all $m <n$. Also observe that if \eqref{ker-growth} holds with $R=1$ as threshold then, up to modifying the constants $C$, it also holds for any $R>0$.\\\\
According to \eqref{kernel-defi}, we have for $i\in\{1,2\},$
$$ K^{m+i}(x,y)=\int\limits_MK^{m}(x,z)K^i(z,y)\;d\mu_g(z).$$Let us consider the cases $d_g(x,y)<1$ and $d_g(x,y) \ge 1$ separately.\\\\

{\it Case 1:} Let $x,y \in M$ be such that $d_g(x,y)<1.$\\
$$ \int\limits_MK^{m}(x,z)K^i(z,y)\;d\mu_g(z) = \int\limits_{B(x,2)}K^{m}(x,z)K^i(z,y)\;d\mu_g(z)$$ $$+ \int\limits_{M\setminus B(x,2)}K^{m}(x,z)K^i(z,y)\;d\mu_g(z)$$
The second integral on the right is uniformly bounded independent of $x$ as it is bounded from above by 
$$ \left(\int\limits_{M\setminus B(x,2)}(K^{m}(x,z))^2\;d\mu_g(z)\right)^{\frac{1}{2}}\left(\int\limits_{M\setminus B(y,1)}(K^i(z,y))^2\;d\mu_g(z)\right)^{\frac{1}{2}}$$
, and using the estimates \eqref{infinity-estimate}, \eqref{gradient-estimate}, and \eqref{infinity-kernel}.\\
Next, we will estimate the first term. First, we will consider the case $i=2$. From \eqref{upper-lower-bound} and the fact that $K^{m}$ satisfies \eqref{ker-growth}, we get
$$ \int\limits_{B(x,2)}K^{m}(x,z)K^2(z,y)\;d\mu_g(z) \le $$ $$\int\limits_{B(x,2)}\alpha_{n,m}\left[d_g(x,z)\right]^{m-n}\left(1+ C\left[d_g(x,z)\right]^{\frac{1}{2}}\right) \Phi(d_g(z,y))d\mu_g(z).$$
We will estimate the right-hand side by writing it in the normal coordinates centered at $x$ as we did in the proof of Theorem \ref{greens-sing-estimate}.
Using the same notation and proceeding as before,  we get
$$ \mathcal{I}:= \int\limits_{B(x,2)}\left[d_g(x,z)\right]^{m-n}\left(1+ C\left[d_g(x,z)\right]^{\frac{1}{2}}\right) \Phi(d_g(z,y))d\mu_g(z) $$ 
$$\le \int\limits_{B(0,2)}|\tilde{z}|^{m-n}\left(1+ C|\tilde{z}|^{\frac{1}{2}}\right) \frac{|\tilde{z}-\tilde y|^{2-n}}{(n-2)\omega_{n-1}}\left(\frac{\sinh (b|\tilde{z}|)}{b|\tilde{z}|}\right)^{n-1}\; d\tilde{z}.$$

For $\tilde{z} \neq 0$, we decompose the integrand as follows

 \begin{multline*}
|\tilde{z}|^{m-n}\left(1+ C|\tilde{z}|^{\frac{1}{2}}\right) \frac{|\tilde{z}-\tilde y|^{2-n}}{(n-2)\omega_{n-1}}\left(\frac{\sinh (b|\tilde{z}|)}{b|\tilde{z}|}\right)^{n-1} = \\
|\tilde{z}|^{m-n}\left(1+ C|\tilde{z}|^{\frac{1}{2}}\right) \frac{|\tilde{z}-\tilde y|^{2-n}}{(n-2)\omega_{n-1}}\left(\left(\frac{\sinh (b|\tilde{z}|)}{b|\tilde{z}|}\right)^{n-1} -1   \right)+ \\
|\tilde{z}|^{m-n}\left(1+ C|\tilde{z}|^{\frac{1}{2}}\right) \frac{|\tilde{z}-\tilde y|^{2-n}}{(n-2)\omega_{n-1}}.
\end{multline*}  

Note that each term above is nonnegative and, for $|\tilde{z}| \leq 2$, 
$$0 \leq  \left(\frac{\sinh (b|\tilde{z}|)}{b|\tilde{z}|}\right)^{n-1} -1 \leq \tilde C|\tilde{z}|^2.$$
Combining these facts together, we obtain
\begin{multline*}
 \mathcal{I} \leq \int\limits_{B(0,2)}   |\tilde{z}|^{m-n} \frac{|\tilde{z}-\tilde y|^{2-n}}{(n-2)\omega_{n-1}}\; d\tilde{z} +  C \int\limits_{B(0,2)}  |\tilde{z}|^{m-n+ 1/2}\frac{|\tilde{z}-\tilde y|^{2-n}}{(n-2)\omega_{n-1}}\; d\tilde{z} \\
+ \hat{C} \int\limits_{B(0,2)}    |\tilde{z}|^{m-n+2}\frac{|\tilde{z}-\tilde y|^{2-n}}{(n-2)\omega_{n-1}}                 \; d\tilde{z}
\end{multline*}  

Bounding each term by integrating over $\R^n$ instead of $ B(0,2)$, 
 and using, for $0<\alpha,\beta < n$ such that $\alpha+\beta <n$, 
$$\int\limits_{\R^n}|x|^{\alpha-n}|x-y|^{\beta-n}\; dx\;=\;\frac{\gamma(\alpha)\gamma(\beta)}{\gamma(\alpha+\beta)} |y|^{\alpha+\beta -n},$$
where 
$$\gamma(x)=2^x \pi^{\frac{n}{2}}\frac{\Gamma(\frac{x}{2})}{\Gamma(\frac{n-x}{2})} $$

(see \cite{Stein}, Chapter 5), we get the estimate in this case.\\
Next, we consider the case $i=1$ where the arguments are similar, and hence we will only outline the proof.\\
$$ \int\limits_{B(x,2)}K^{m}(x,z)K^1(z,y)\;d\mu_g(z) \le \int\limits_{B(x,2)}K^{m}(x,\cdot)|\nabla_g \Phi(d_g(\cdot,y))|\;d\mu_g$$
$$ + \int\limits_{B(x,2)}K^{m}(x,\cdot)|\nabla_g G(\cdot,y) -\nabla_g \Phi(d_g(\cdot,y))|\;d\mu_g  = \ \mathcal{I} \ + \ \mathcal{II}. $$ 
We can proceed exactly as in the case of $i=2$ to estimate $\mathcal{I}$ and we see that we get the exact constant $\alpha_{n,m+1}$. While $\ \mathcal{II} $ can be estimated by using \eqref{gradient-singularity} to get
 $$\mathcal{II} \le \left\{\begin{array}{ll}
   C\int\limits_{B(x,2)}\left[d_g(x,z)\right]^{m-n} \left[d_g(z,y)\right]^{3-n}d\mu_g(z) & {\rm if} \ n>3,\\
   C\int\limits_{B(x,2)}\left[d_g(x,z)\right]^{m-n} \left[1+|\log d_g(z,y)|\right]d\mu_g(z) & {\rm if} \ n=3.
\end{array} \right. $$
When $n=3$, the only possible value of $m$ to be considered is $m=2$, but $m+1=3=n$ and hence we have to consider only $n>3$. As estimated before, we can easily see that $\mathcal{II} \le C \left[d_g(x,y)\right]^{m+\frac{3}{2}-n}$, and this completes the estimates of {\it Case 1}.\\\\
{\it Case 2:} Let  $x,y \in M$ be such that $d_g(x,y)\ge 1.$\\\\
Let us denote $d:= d_g(x,y)$, then
$$ \int\limits_MK^{m}(x,z)K^i(z,y)\;d\mu_g(z) = \int\limits_{B(y,\frac{d}{2})}K^{m}(x,z)K^i(z,y)\;d\mu_g(z) $$ $$+ \int\limits_{M\setminus B(y,\frac{d}{2})}K^{m}(x,z)K^i(z,y)\;d\mu_g(z)$$
Since $K^{m}$ satisfies the lemma, we get using \eqref{local-estimate} and \eqref{local-gradient-estimate}
$$ \int\limits_{B(y,\frac{d}{2})}K^{m}(x,z)K^i(z,y)\;d\mu_g(z) \le C e^{-\alpha_m \frac{d}{2}}\int\limits_{B(y,\frac{d}{2})}K^i(z,y)\;d\mu_g(z)\le C e^{-\alpha_m\frac{d}{4}}$$
where $C,\alpha_m $ are independent of $x$ and $y$.
Now
$$\int\limits_{M\setminus B(y,\frac{d}{2})}K^{m}(x,z)K^i(z,y)\;d\mu_g(z)=\int\limits_{B(x,\frac{1}{2})}K^{m}(x,z)K^i(z,y)\;d\mu_g(z)$$ 
$$+\int\limits_{M\setminus \left[B(y,\frac{d}{2}) \cup B(x,\frac{1}{2})\right]}K^{m}(x,z)K^i(z,y)\;d\mu_g(z)
\le C\Psi_a(\frac{d}{2}) \int\limits_{B(x,\frac{1}{2})}K^{m}(x,z)\;d\mu_g(z)$$
$$+  \left(\int\limits_{M \setminus B(x,\frac{1}{2})}(K^{m}(x,z))^2\;d\mu_g(z)\right)^{\frac{1}{2}} \left(\int\limits_{M\setminus B(y,\frac{d}{2})}(K^i(z,y))^2\;d\mu_g(z) \right)^{\frac{1}{2}}$$
Using \eqref{infinity-estimate}, \eqref{gradient-estimate}, and \eqref{infinity-kernel}, we get a bound of the form $C e^{-\beta d}$ for the last term in the above inequality for some positive constants $\beta, C$ independent of $x,y.$  Next, we show that  $ \int\limits_{B(x,\frac{1}{2})}K^{m}(x,z)\;d\mu_g(z)$ is bounded independent of $x$. Since $K^{m}$ satisfies the lemma, writing in the normal coordinates centered at $x$, and using Theorem \ref{volumeelement}  we get
$$ \int\limits_{B(x,\frac{1}{2})}K^{m}(x,z)\;d\mu_g(z) \le C\int\limits_0^\frac{1}{2} \int\limits_{S^{n-1}}r^{m-n}r^{n-1}A_x(r,\theta)dr\;d\theta $$ $$\le C\int\limits_0^\frac{1}{2} r^{m-n} (\sinh (br))^{n-1}dr \le C.$$
Combining all the above estimates, we see that \eqref{ker-growth} holds for $K^{m+i}.$ \\\\
It remains to show that \eqref{infinity-kernel} holds for $K^{m+i}$. First observe that from \eqref{gradient-point-est}, and for $m$ even, we have $K^{m+1}(x,y) \le CK^{m+2}(x,y)$ when $d_g(x,y)>1$ where the constant is uniform in $x,y$. Thus, it is enough to establish \eqref{infinity-kernel} holds for $K^{m+2}$.\\\\ For this purpose, let us define $K^m_R(x,y)$ for $x,y\in B_R,\;x\not=y$ as in \eqref{kernel-defi} with $G_R$ instead of $G$, where $G_R$ is as in the Proof of Theorem \ref{greens}. Then using the monotone convergence theorem, we see that for all $x\not=y$, $K_R^m(x,y)\rightarrow K^m(x,y)$ as $R\rightarrow \infty$ and for any fixed $x\in M$, $K_R^m(x,\cdot)$ solves\\
 $$-\Delta_g K_R^{m+2}(x,\cdot)= K_R^{m}(x,\cdot),\;\; \mbox{ and } K_R^{m+2}(x,y)=0\;\; \mbox{ for }y\in \partial B_R .$$ 
Let $f\in C^1(M)$ be such that $ 0\le f\le 1$, and $f=0$ in a neighbourhood of $x$.\\
Multiplying the above equation by $f^2K_R^{m+2}$, we get
\begin{equation}\label{eq:estiLem41}\int\limits_{B_R} -\Delta_g K_R^{m+2}(x,y) (f(y))^2K_R^{m+2}(x,y) d\mu_g(y)= \int\limits_{B_R}K_R^{m}(x,y)K_R^{m+2}(x,y) (f(y))^2\;d\mu_g(y).
\end{equation}

The term on the left-hand side can be rewritten as
\begin{multline*} \int\limits_{B_R} -\Delta_g K_R^{m+2}(x,y) (f(y))^2K_R^{m+2}(x,y) d\mu_g(y)=\\
 \int\limits_{B_R} |\nabla_g(f(y) K_R^{m+2}(x,y))|^2 d\mu_g(y) - 
\int\limits_{B_R} |\nabla_gf|^2 (K_R^{m+2}(x,y))^2 d\mu_g(y).\end{multline*}

Inserting this into \eqref{eq:estiLem41}, we  obtain
\begin{eqnarray*} & &\int\limits_{B_R} |\nabla_g(f(y) K_R^{m+2}(x,y))|^2 d\mu_g(y) - 
\int\limits_{B_R} |\nabla_gf|^2 (K_R^{m+2}(x,y))^2 d\mu_g(y)  \\
&\le& \left(\int\limits_{B_R}(K_R^{m}(x,y)f(y))^2\;d\mu_g(y)\right)^{\frac{1}{2}}\left(\int\limits_{B_R}(K_R^{m+2}(x,y) f(y))^2\;d\mu_g(y)\right)^{\frac{1}{2}} \\
&\le& C\int\limits_{B_R}(K_R^{m}(x,y)f(y))^2\;d\mu_g(y) + \frac{1}{2}\left(\frac{(n-1)a}{2}\right)^2 \int\limits_{B_R}(K_R^{m+2}(x,y) f(y))^2\;d\mu_g(y), 
\end{eqnarray*}
where we apply Young's inequality to get the last line. Using Theorem \ref{mckean} and taking the limit $R\rightarrow \infty$, we get
\begin{equation}\label{cutoff}
\begin{array}{ll}
\frac{1}{2}\left(\frac{(n-1)a}{2}\right)^2 \int\limits_M (f(y) K^{m+2}(x,y))^2 d\mu_g(y) - 
 & \int\limits_{M}|\nabla_gf|^2  (K^{m+2}(x,y))^2 d\mu_g(y)\\
 & \le C\int\limits_{M}(K^{m}(x,y)f(y))^2\;d\mu_g(y)
\end{array}
\end{equation}
Taking $f$ such that $f=0$ in $B(x,\frac{1}{2})$ and $f=1$ in $M\setminus B(x,1)$, we get
$$\int\limits_{M\setminus B(x,1)} (K^{m+2}(x,y))^2 d\mu_g(y)\le C \int\limits_{B(x,1)\setminus B(x,\frac{1}{2})} (K^{m+2}(x,y))^2 d\mu_g(y)$$
$$ +C \int\limits_{M\setminus B(x,\frac{1}{2})} (K^{m}(x,y))^2 d\mu_g(y). $$
The first term on the right-hand side of the above inequality is bounded independently of $x$ as $K^{m+2}(x,y) \le C (d_g(x,y))^{m+2-n}$ and the measure of the annulus is bounded independently of $x$ thanks to the lower bound on the Ricci curvature. The second term is bounded by assumption. Thus, there exists a $C>0$ such that for all $x\in M$
\begin{equation}\label{induction-bound}
\int\limits_{M\setminus B(x,1)} (K^{m+2}(x,y))^2 d\mu_g(y)\le C. 
\end{equation}  

Let $R>0$ and choose $f_R \in C^1(M)$ such that
$$f_R=0 \;{\rm in}\; B(x,R),\; f_R=1 \;{\rm in}\; M \setminus B(x,R+1),\; |\nabla_gf_R|\le 1,\; 0\le f_R\le 1. $$
Then, by taking $f=f_R$ in \eqref{cutoff}, and using the fact that 
$$B(x,R+1)\setminus B(x,R) = \left(M \setminus B(x,R)\right)\setminus \left( M \setminus B(x,R+1)\right), $$
the equation \eqref{cutoff} simplifies to
$$ \left[\frac{1}{2}\left(\frac{(n-1)a}{2}\right)^2 +1 \right]\int\limits_{M\setminus B(x,R+1)} (K^{m+2}(x,y))^2 d\mu_g(y)$$ 
$$\le \int\limits_{M\setminus B(x,R)} (K^{m+2}(x,y))^2 d\mu_g(y)+ C\int\limits_{M\setminus B(x,R)}(K^{m}(x,y))^2\;d\mu_g(y).$$
Thus, if we denote $\alpha = \left[\frac{1}{2}\left(\frac{(n-1)a}{2}\right)^2 +1 \right]^{-1} $, then $0<\alpha <1$ and it satisfies for all $R>0$,
$$ \int\limits_{M\setminus B(x,R+1)} (K^{m+2}(x,y))^2 d\mu_g(y)\le \alpha\int\limits_{M\setminus B(x,R)} (K^{m+2}(x,y))^2 d\mu_g(y) $$
$$+ C\alpha\int\limits_{M\setminus B(x,R)}(K^{m}(x,y))^2\;d\mu_g(y).$$
Now, let $R>1$, then $k\le R <k+1$ for some $k\in \mathbb{N}$, and a repeated use of the above inequality gives
$$\int\limits_{M\setminus B(x,R)} (K^{m+2}(x,y))^2 d\mu_g(y)\le \int\limits_{M\setminus B(x,k)} (K^{m+2}(x,y))^2 d\mu_g(y)$$
$$ \le \alpha^{k-1}\int\limits_{M\setminus B(x,1)} (K^{m+2}(x,y))^2 d\mu_g(y)\; +\;\sum\limits_{i=1}^{k-1}C \alpha^i\int\limits_{M\setminus B(x,k-i)}(K^{m}(x,y))^2\;d\mu_g(y) $$
$$ \le \alpha^{R-2}\int\limits_{M\setminus B(x,1)} (K^{m+2}(x,y))^2 d\mu_g(y)\; +\;\sum\limits_{i<\frac{k}{2}}C \alpha^i\int\limits_{M\setminus B(x,k-i)}(K^{m}(x,y))^2\;d\mu_g(y) $$
$$+ \sum\limits_{i\ge\frac{k}{2}}^{k-1}C \alpha^i\int\limits_{M\setminus B(x,k-i)}(K^{m}(x,y))^2\;d\mu_g(y)$$
$$\le  \alpha^{R-2}\int\limits_{M\setminus B(x,1)} (K^{m+2}(x,y))^2 d\mu_g(y)+ Ck\int\limits_{M\setminus B(x,\frac{k}{2})}(K^{m}(x,y))^2\;d\mu_g(y) $$ $$ + Ck\alpha^{\frac{k}{2}}\int\limits_{M\setminus B(x,1)}(K^{m}(x,y))^2\;d\mu_g(y) $$ 
$$\le C e^{(R-2)log\alpha} + Cke^{-\alpha_m\frac{k}{2}} + Ck \alpha^{\frac{k}{2}}\le C e^{-\alpha_{m+2} R}$$
for some ${\alpha}_{m+2}>0$, thanks to \eqref{induction-bound}.
\end{proof}
\subsection{Symmetrization of the kernel.} Recall, for a function $f:M\rightarrow [-\infty, \infty ]$ the distribution function of $f$ is given by
$$\lambda_f(t) \; = \; \mu_g(\{x\in M : |f(x)|>t\})\;,\; t\in \R,$$
and its nonincreasing rearrangement $f^\ast :(0,\infty) \rightarrow (0,\infty)$
is defined by
$$f^\ast(t)\;=\; \inf\{s : \lambda_f(s)\le t \}\; , \; t>0.  $$
For $K : M\times M \rightarrow [-\infty, \infty]$, denote by $K^x$ the function $y\rightarrow K(x,y)$. Denote by $K^\ast $ and $K^{\ast\ast} $ the functions
$$K^\ast(t)\; =\; \sup\limits_{x\in M}(K^x)^\ast(t)\;\;,\;K^{\ast\ast}(t)\; = \;\frac{1}{t}\int\limits_0^t K^\ast (s)\; ds \ , \  t>0.$$
We have the following estimate on the kernel $K^m$ introduced in \eqref{kernel-defi}.\\
\begin{thm}\label{symme-estimates} Let $(M,g),\; K^m$ be as in Lemma \ref{kernel-estimate} then
\begin{enumerate}
\item[(i)]there exist constants $A,\beta >0$ such that
\begin{equation}\label{symme-zero}
(K^m)^\ast(t)\  \le \ \left[\beta_0(m,n)t \right]^{\frac{m-n}{n}}\left[1+ At^\beta\right]  \ \ {\text for} \  0<t\le 1. 
\end{equation}
\item[(ii)] For any $\sigma \in (0,1)$, there exists $B_{\sigma}>0$ such that
\begin{equation}\label{symm-infty}
(K^m)^\ast(t)\  \le \ \frac{B_\sigma}{t^\sigma}  \ \ {\text for} \  t > 1. 
\end{equation}
\end{enumerate}
\end{thm}
\begin{proof} First note that if  $f(t) = At^{-\alpha}\left[ 1+Bt^\beta\right],\; t > 0$, for positive constants $A,B,\alpha, \beta $ such that $\beta <\alpha$, then there exists a $C>0$ such that  
$$ f^{-1}(t) \le \left[At^{-1}\right]^\frac{1}{\alpha}\left[ 1+ C t^{-\frac{\beta}{\alpha}}\right] \ \ {\text for} \ \ \  t > 1. $$
Using this together with \eqref{ker-growth} and Theorem \ref{volume}, we get for $t>1$,
$$
\mu_g(\{y\in M : K^m(x,y)>t\}) \le \mu_g(B(x,f^{-1}(t))) \le V^n_{-b^2}(f^{-1}(t)), 
$$
where $f$ is as above with $A=\alpha_{n,m}$, $\alpha= n-m$, $B =C$, and $\beta = \frac{1}{2}$. Now, substituting $V^n_{-b^2}(f^{-1}(t))$ using \eqref{volume-model}, we get for any $x\in M,$
$$\mu_g(\{y\in M : K^m(x,y)>t\}) \ \le \ \frac{\omega_{n-1}}{n}\left(\frac{\alpha_{n,m}}{t} \right)^{\frac{n}{n-m}}\left[1+ C t^{\frac{-1}{2(n-m)}} \right] \ \ {\text for} \ \ \  t > 1.$$
Again, if   $g(t) = At^{-\alpha}\left[ 1+Bt^{-\beta}\right],\; t > 0$, for positive constants $A,B,\alpha, \beta $, then there exists a $C>0$ such that  
$$ g^{-1}(t) \le \left[At^{-1}\right]^\frac{1}{\alpha}\left[ 1+ C t^{\frac{\beta}{\alpha}}\right] \ \ {\text for} \ \ \  0<t \le 1. $$
Using this fact together with the above estimate proves \eqref{symme-zero}.\\
To prove \eqref{symm-infty}, first recall from \eqref{sublevel-measure} and \eqref{sublevel-infinity} we have for any $x\in M$,
\begin{equation}\label{truc}
\mu_g(\{y\in M : G^x(y)>t\}) \ \le \begin{cases}
\ \frac{C}{t} \ \ {\text for} \ \ \ 0<t\le 1, \\
\ \frac{C}{t^{\frac{n}{n-2}}} \ \ {\text for} \ \ \ t > 1, 
\end{cases}
\end{equation}
where $C$ is independent of $x$. Hence
\begin{equation}\label{gstar}
G^\ast(t) \ \le \begin{cases}
\ \frac{C}{t^{\frac{n-2}{n}}} \ \ {\text for} \ \ \ 0<t\le 1, \\
\  \frac{C}{t}\ \ {\text for} \ \ \ t > 1. 
\end{cases}
\end{equation}
This immediately proves \eqref{symm-infty} when $m=2.$ 

We need similar estimates for $|\nabla_gG^x|^\ast(t)$. To get them, we combine the pointwise gradient estimate \eqref{gradient-point-est} together with the bounds \eqref{eucli-estimate} and \eqref{ricci-estimate} on the Green function. Using \eqref{truc}, we derive an upper bound for $\mu_g(\{y\in M : |\nabla_gG^x|(y)>t\})$, similar to \eqref{truc},  when $t$ is large or close to $0$. More precisely, up to modifying the constants, we get

\begin{equation}\label{dgstar}
|\nabla_gG^x|^\ast(t) \ \le \begin{cases}
\ \frac{C}{t^{\frac{n-1}{n}}} \ \ {\text for} \ \ \ 0<t\le 1 \\
\  \frac{C}{t}\ \ {\text for} \ \ \ t > 1 
\end{cases}
\end{equation}
uniformly in $x$. Combining \eqref{gstar} and \eqref{dgstar} we have for $i =1,2$
\begin{equation}\label{istar}
(K^i)^\ast(t) \ \le \begin{cases}
\ \frac{C}{t^{\frac{n-i}{n}}} \ \ {\text for} \ \ \ 0<t\le 1, \\
\  \frac{C}{t}\ \ {\text for} \ \ \ t > 1. 
\end{cases}
\end{equation}

 Now assume the result is true for some even integer $m.$ We claim that it will be true for $m+i$ if $m+i<n$, where $i=1,2.$\\ 
Fix $x\in M$, then
$$ K^{m+i}(x,y) = \int\limits_M K^m(x,z)K^i(z,y) d\mu_g(z)= \int\limits_M K^i(y,z)(K^m)^x(z)  d\mu_g(z)$$
i.e., for $x\in M$, $(K^{m+i})^x$ is obtained by integrating $(K^m)^x$ against the kernel $K^i$. Thus, it follows from the improved version of O'Neil's lemma (see \cite[Lemma2]{FonMo}) that
$$[(K^{m+i})^x]^\ast(t) \le [(K^{m+i})^x]^{\ast\ast}(t) \le t (K^i)^{\ast\ast}(t) [(K^{m})^x]^{\ast\ast}(t) + \int\limits_t^\infty (K^i)^{\ast}(s)[(K^{m})^x]^{\ast}(s)\; ds. $$
Now, the estimate \eqref{symm-infty} on $K^{m+i}$ follows from the induction assumption and \eqref{istar}.
\end{proof}
\subsection{Proof of theorem.} As stated before we will prove our theorem by writing the functions as integrals of the corresponding derivatives against kernels, thus following an idea initiated in \cite{A}, and developed further by Fontana and collaborators. Let us recall the following theorem which is essentially \cite[Theorem 3]{FontaMo}.\\
\begin{thm}\label{thmFM} Let $(M,g)$ be a Hadamard manifold and $K : M\times M  \rightarrow [-\infty,\infty]$ be a measurable function satisfying
$K(x,y)= K(y,x)$, for all $x,y$, and for some $1<q <\infty$,
\begin{equation}\label{ker-condition1}
K^\ast(t) \le \begin{cases} \left[At \right]^{\frac{-1}{q^\prime}}\left[1+ Ct^\beta\right]  \ \ {\text for} \ \ \ 0<t\le 1, \\
B t^{\frac{-1}{q^\prime}}  \ \ {\text for} \ \ \  t > 1, 
\end{cases}
\end{equation} 
and
\begin{equation}\label{ker-condition2}
\int\limits_1^\infty(K^\ast(t))^{q^\prime}\; dt < \infty,
\end{equation}
where $q^\prime = \frac{q}{q-1}$ and $\beta, A,B,C$ are fixed positive constants. For a measurable function $f:M\rightarrow \R$, define for $x \in M$,
\begin{equation}\label{int-operator}
Tf(x) \;= \; \int\limits_MK(x,y)f(y)\; d\mu_g(y),
\end{equation}
whenever the integral exists. Then, $Tf(x)$ is defined for a.e. $x\in M$ when $f\in L^q(M)$ and there exists a constant $\tilde C >0$ such that 
\begin{equation}\label{conclusion}
\int\limits_N \exp \left({A |Tf(x)|^{q^{\prime}}}\right) \ d\mu_g(x) \;\le \tilde{C}\left(1+ \mu_g(N)\right) 
\end{equation}
holds for all measurable subsets $N$ of $M$ with $\mu_g(N) < \infty$ and $f\in L^q(M)$ with $||f||_q\le 1$.
\end{thm}

Using the above theorem and the estimates on the kernels developed in the previous section, we can now prove our main result.\\ 
\noindent {\bf Proof of Theorem \ref{theorem1}:} First note that a repeated use of \eqref{repre-1} gives  
$$ u(x) \ = \ (-1)^{\frac{k}{2}}\int\limits_M K^k(x,y)\nabla_g^k u(y)\ d\mu_g(y),\ \ u\in C_c^k(M)$$
when $k\in \mathbb{N}$ and $k$ is even. When $k$ is odd, applying the above result for $k+1$ and then integrating by parts gives
$$ u(x) \ = \ (-1)^{\frac{k-1}{2}}\int\limits_M \langle \nabla_gK^{k+1}(x,\cdot),\nabla_g^k u(\cdot)\rangle_g\ d\mu_g(y),\ \ u\in C_c^{k+1}(M).$$
Also we have from \eqref{kernel-defi} and for $k$ odd, $$|\nabla_gK^{k+1}(x,\cdot)| \ \le \ K^{k}(x,\cdot). $$
Combining these facts we get
\begin{equation}\label{u-rep}
|u(x)| \ \le T^k(|\nabla_g^k u|),
\end{equation} 
for $u\in C_c^{k+1}(M)$ and hence for $u\in C_c^{k}(M)$ by approximation, where $T^k$ is defined as in \eqref{int-operator} with $K(x,y)= K^k(x,y)$ when
$x\not=y$ and $K(x,x)=0.$ Moreover, from Theorem \ref{symme-estimates}, we see that $K^k$ satisfies the assumptions of the above theorem with $q=p=\frac{n}{k}$ and $A= \beta_0(k,n)$.  Thus, Theorem \ref{thmFM} applies and we get for  $u\in C_c^k(M)$ with $\int\limits_M|\nabla_g^ku|^p \ d\mu_g \ \le 1$, and $N\subset M$ with $\mu_g(N) < \infty$,
$$\int\limits_N\exp\left({\beta_0(k,n)|u(x)|^{p^{\prime}}}\right) \ d\mu_g(x)  \le \tilde{C}\left(1+ \mu_g(N)\right). $$
If $u\in C_c^k(M)$, then $$\mu_g(\{x\in M: |u(x)|> 1\}) \le ||u||_p^p.$$
Taking $N$ as this set, we get for all $u\in C_c^k(M)$ with $\int\limits_M |\nabla_g^ku|^{p} \, d\mu_g \, \le 1$:
$$\int\limits_{\{x:|u(x)|>1\}}\exp\left({\beta_0(k,n)|u(x)|^{p^{\prime}}}\right) \ d\mu_g(x)  \le \tilde{C}\left(1+||u||_p^p\right). $$
Thus, for $u\in C_c^k(M)$ with $\int\limits_M |\nabla_g^k u|^{p} \, d\mu_g \ \le 1$, 

\begin{equation*}
\begin{split}
\int\limits_ME_{[p-1]}\left({\beta_0(k,n)|u(x)|^{p^{\prime}}}\right) \ d\mu_g \; =  \hspace*{1cm}& \\
 \int\limits_{\{x:|u(x)|\le 1\}}E_{[p-1]}\left({\beta_0(k,n)|u(x)|^{p^{\prime}}}\right) \ d\mu_g +& \int\limits_{\{x:|u(x)|>1\}}E_{[p-1]}\left({\beta_0(k,n)|u(x)|^{p^{\prime}}}\right) \ d\mu_g \\
 \le C\int\limits_{\{x:|u(x)|\le 1\}}|u(x)|^{p} \ d\mu_g &+ \int\limits_{\{x:|u(x)|>1\}}\exp\left({\beta_0(k,n)|u(x)|^{p^{\prime}}}\right) \ d\mu_g \\
 \le \tilde{C}\Big( 1+  \int\limits_M |u|^p&\ d\mu_g(x)\Big) 
 \end{split}
\end{equation*}

We get from Theorem \ref{mc-st-k} that if $u\in C_c^k(M)$ with $\int\limits_M |\nabla_g^ku|^p \, d\mu_g \ \le 1$, then $\int\limits_M |u|^p \, d\mu_g \ \le C_{k,p}$. Hence the conclusion of the theorem follows.

The optimality of the constant $\beta_0(k,n)$ follows using standard test functions (see \cite{A} for the proof in the Euclidean case and \cite[Proposition 3.6]{Fonta} for the Riemannian case). This completes the proof.
 \qed\\
 
\noindent{\bf Proof of Theorem \ref{theorem2}.} Let  $s$ be an integer satisfying $s \geq p-1$ . Then, it follows from Theorem \ref{theorem1} that there exists $C>0$ such that 
$$\frac{1}{s!}\int_{M} \left(\beta_0(k,n) |u(x)|^{p^{\prime}}\right)^s \ 
d\mu_g \ \le C  \left[\int\limits_M |\nabla_g^ku|^p  \ d\mu_g\right]^{\frac{p^\prime s}{p}} , \forall \ u \in C_c^k(M).$$
This immediately gives for $q=p^\prime s$,
$$ S_q \ge \frac{(\beta_0(k,n))^{\frac{p}{p^\prime}}}{(C s!)^{\frac{p}{q}}}$$
Thus, using interpolation if $ q =(1-\theta) p^\prime s + \theta p^\prime (s+1) ,\ \theta\in (0,1),$ we get 

$$ S_q \ge \frac{(\beta_0(k,n))^{\frac{p}{p^\prime}}}{(C s!)^{\frac{p}{q}}(s+1)^{\frac{p\theta}{q}}}$$
Taking the limit  $q\rightarrow \infty$ using Sterling's formula we get
$$\liminf_{q\rightarrow \infty}\left[ q^{p-1}S_q \right] \ \geq \ \left[\frac{p}{p-1} e \,\beta_0(k,n)\right]^{p-1}.$$
It remains to show that
$$\limsup_{q\rightarrow \infty}\left[ q^{p-1}S_q \right] \ \leq \ \left[\frac{p}{p-1} e \beta_0(m,n)\right]^{p-1}.$$

We will prove this inequality by using the test functions used by Adams \cite{A} to establish the best constant in Adams inequality in the Euclidean space. In fact we will use it by lifting to the manifold as done in \cite{Fonta}.\\
Let $\Phi :[0,1]\rightarrow \R$ be a $C^\infty$ function such that
$$\Phi(0) = \Phi^\prime(0) = \cdots = \Phi^{k-1}(0) = 0$$ and 
$$\Phi(1) = \Phi^\prime(1) = 1,\ \Phi^{\prime\prime}(1)= \cdots = \Phi^{k-1}(1) = 0. $$
For $\epsilon \in (0,\frac{1}{2})$ define 
$$
H(t) = \epsilon \Phi(\frac{t}{\epsilon})\  \chi_{[0,\epsilon]}(t)\ +\ t \chi_{(\epsilon,1-\epsilon]}(t) \ + \ \left( 1- \epsilon \Phi(\frac{1-t}{\epsilon})\right)\ \chi_{(1-\epsilon, 1]}(t) \ +\ \chi_{(1,\infty)}.$$ 
For $R\in (0,1)$, define the function $\psi_R$ by
$$ \psi_R(t)\ = \ H \left(\left(\log \frac{1}{R}\right)^{-1}\log \frac{1}{t} \right).$$
Fix $x_0 \in M$. For  $\epsilon \in (0,\frac{1}{2})$ and $R\in (0,1)$, define $u_R : M \rightarrow \R$ by
$$ u_R(x)\ = \ \psi_R(d_g(x_0,x)), \ x\in M.$$
Then $u_R \in C_c^k(M)$ with support in $B(x_0,1)$, and $u_R =1$ on $B(x_0,R)$.
Moreover, we have from the computations of \cite{Fonta} and \cite{A}
$$\int\limits_M|\nabla_g^k u_R|^pd\mu_g = \left(\omega_{n-1}\log \frac{1}{R}\right)^{1-p}\alpha(k,n)^p \left[ 1+ C\epsilon + O\left((\log \frac{1}{R})^{-1}\right)\right]$$
as $R\rightarrow 0$ where $C$ is independent of $\epsilon$ and $R$ and 
$$
\alpha(k,n) \ = \ \left\{ \begin{array}{ll}
\omega_{n-1} 2^{\frac{k-2}{2}}\Gamma (\frac{k}{2})\ (n-k)(n-k+2)\cdots (n-2) 
&{\rm if} \ \ k \ {\rm even}, \\
\omega_{n-1} 2^{\frac{k-1}{2}}\Gamma (\frac{k+1}{2})\ (n-k+1)(n-k+3)\cdots (n-2) &{\rm if} \ \ k \ {\rm odd}. 
\end{array}
\right.
$$ 
Now
$$S_q \le \frac{\int\limits_M|\nabla_g^ku_R|^p   \ d\mu_g}{\left[\int\limits_M |u_R|^{q}\ d\mu_g\right]^{\frac{p}{q}}}   \le \frac{\int\limits_M|\nabla_g^ku_R|^p  \ d\mu_g}{\left[\int\limits_{B(x_0,R)} |u_R|^{q}\ d\mu_g\right]^{\frac{p}{q}}} $$
$$\le \frac{\left(\omega_{n-1}\log \frac{1}{R}\right)^{1-p}\alpha(k,n)^p \left[ 1+ C\epsilon + O\left((\log \frac{1}{R})^{-1}\right)\right]}{\left(\frac{\omega_{n-1}R^n}{n}(1\ + O(R))\right)^{\frac{p}{q}}}$$
$$= \frac{1}{q^{p-1}}\left[(1+ C\epsilon)\left(\frac{p}{p-1} e \beta_0(k,n)\right)^{p-1}\ + \ o(1)\right]$$
if we set $ \log \frac{1}{R}= \frac{p-1}{pn}q$ as $q \rightarrow \infty$. Taking $\epsilon \rightarrow 0$, we get the required assertion and this
completes the proof.
\qed

\bibliographystyle{plain}


\begin{thebibliography}{99}
\bibitem{A} D.R. Adams. {\it A sharp inequality of J. Moser for higher order derivatives}, Ann. of Math. (2), 128 (2) (1988), pp. 385-398



\bibitem{AdiT} A. Adimurthi and K. Tinterev. {\it On a version of Trudinger-Moser inequality with M{\"o}bius shift invariance}, Calc. Var. Partial Differential Equations 39 (2010), no. 1-2, 203-212


\bibitem{Aubin}  T. Aubin. {\it Some nonlinear problems in Riemannian geometry}, Springer Monographs in Mathematics. Springer-Verlag, Berlin, 1998.

\bibitem{Auscher}
P. Auscher, T. Coulhon, X.~T. Duong, and S. Hofmann.
\newblock Riesz transform on manifolds and heat kernel regularity.
\newblock {\em Ann. Sci. \'{E}cole Norm. Sup. (4)}, 37(6):911--957, 2004.

\bibitem{Bakry_Riesz}
D. Bakry.
\newblock \'{E}tude des transformations de {R}iesz dans les vari\'{e}t\'{e}s
  riemanniennes \`a courbure de {R}icci minor\'{e}e.
\newblock In {\em S\'{e}minaire de {P}robabilit\'{e}s, {XXI}}, volume 1247 of
  {\em Lecture Notes in Math.}, pages 137--172. Springer, Berlin, 1987.
  
\bibitem{BM}  L. Battaglia, G. Mancini. {\it Remarks on the Moser-Trudinger inequality}, Adv. Nonlinear Anal. 2 (2013), no. 4, 389-425

\bibitem{Cao} D. Cao. {\it Nontrivial solution of semilinear elliptic equations with critical exponent in $\rnn$}, Communications in Partial Differential Equations, vol. 17, 407-435, 1992.

\bibitem{chavel}  I. Chavel. {\it Eigenvalues in Riemannian geometry}, Pure and Applied Mathematics, 115. Academic Press, Inc., Orlando, FL, 1984.

\bibitem{Che}  P. Cherrier. {\it Une in\'{e}galit\'{e} de Sobolev sur les vari\'{e}t\'{e}s riemanniennes} Bull. Sci. Math. (2) 103 (1979), no. 4, 353-374.

\bibitem{TCOUL} T. Coulhon, X. T. Duong. {\it Riesz transform and related inequalities on noncompact Riemannian manifolds}, Comm. Pure Appl. Math. 56 (2003), no. 12, 1728-1751. 

\bibitem{DoO} J.M. do \'{O}.\emph{N-Laplacian equations in $\rn$ with critical growth,} Abstract and Applied Analysis, vol. 2, pp. 301-315 (1997).

\bibitem{Fonta}  L. Fontana. {\it Sharp borderline Sobolev inequalities on compact Riemannian manifolds}, Comment. Math. Helv. 68 (1993), no. 3, 415-454.

\bibitem{FonMo}  L. Fontana, C. Morpurgo. {\it Adams inequalities on measure spaces}, Adv. Math. 226 (2011), no. 6, 5066-5119. 

\bibitem{FontaMo}L. Fontana, C. Morpurgo. \emph{Sharp Adams and Moser-Trudinger inequalities on $\mathbb{R}^n$ and other spaces of infinite measure},Preprint, arXiv:1504.04678 [math.AP]

\bibitem{gallot}  S. Gallot, D. Hulin, J. Lafontaine. {\it Riemannian geometry}, Third edition. Universitext. Springer-Verlag, Berlin, 2004.

\bibitem{Hebey}  E. Hebey. {\it Nonlinear analysis on manifolds: Sobolev spaces and inequalities}, Courant Lecture Notes in Mathematics, 5. New York University, Courant Institute of Mathematical Sciences, New York; American Mathematical Society, Providence, RI, 1999. 

\bibitem{KaSa}  D. Karmakar, K. Sandeep. {\it Adams inequality on the hyperbolic space}, J. Funct. Anal. 270 (2016), no. 5, 1792-1817.

\bibitem{Komatsu}
H. Komatsu.
\newblock Fractional powers of operators.
\newblock {\em Pacific J. Math.}, 19:285--346, 1966.

\bibitem{LiRuf} Y. Li, B. Ruf. {\it A sharp Trudinger-Moser type inequality for unbounded domains in $\mathbb{\rn}$},
Indiana University Mathematics Journal, vol. 57, no. 1, 451-480, 2008.


\bibitem{LT} P. Li, L.-F. Tam. {\it Symmetric {G}reen's functions on complete manifolds}, Amer. J. Math., 109(6):1129 -1154, 1987.

\bibitem{LW} P. Li, J. Wang. {\it Complete manifolds with positive spectrum}, II. J. Differential Geom. 62 (2002), no. 1, 143-162

\bibitem{LuT} G. Lu, H. Tang {\it  Best constants for Moser-Trudinger inequalities on high dimensional hyperbolic spaces}, 
Adv. Nonlinear Stud. 13 (2013), no. 4, 1035-1052.

\bibitem{LLY} J. Li, G. Lu, Q. Yang. {\it Fourier analysis and optimal Hardy-Adams inequalities on hyperbolic spaces of any even dimension}, Adv. Math. 333 (2018), 350-385.

\bibitem{MS} G. Mancini, K. Sandeep. {\it Moser-Trudinger inequality on conformal discs}, 
Communications in Contemporary Mathematics, vol. 12, no. 6, 1055-1068, 2010.

\bibitem{MST} G. Mancini, K. Sandeep, C. Tintarev. {\it  Trudinger-Moser inequality in the hyperbolic space $\hn$},
Adv. Nonlinear Anal. 2 (2013), no. 3, 309-324.

\bibitem{Moser}  J, Moser. {\it A sharp form of an inequality by N. Trudinger}, Indiana Univ. Math. J. 20 (1970/71), 1077-1092.

\bibitem{Panda} R. Panda. {\it Nontrivial solution of a quasilinear elliptic equation with critical growth in $\mathbb{\rn}$},
Proceedings of the Indian Academy of Science, vol. 105, pp. 425-444, 1995.

\bibitem{petersen}  P. Petersen. {\it Riemannian geometry}, Second edition. Graduate Texts in Mathematics, 171. Springer, New York, 2006.

\bibitem{Pohozaev} S.I. Pohozhaev. {\it The Sobolev imbedding in the case $pl= n$}, Proc.Tech.Sci. Conf.  on  Adv.  Sci.  Research  1964-1965
,  Mathematics  Section,  Moskov.  Energet.
Inst., Moscow (1965), 158-170

\bibitem{Ruf} B. Ruf. {\it A sharp Trudinger-Moser type inequality for unbounded domains in $\mathbb{R}^2$},
Journal of Functional Analysis, vol. 219, no. 2, pp. 340-367, 2005.

\bibitem{Stein} E. M. Stein. {\it Singular integrals and differentiability properties of functions}, Princeton Mathematical Series, No. 30 Princeton University Press, Princeton, N.J.

\bibitem{Stri} R. S. Strichartz. {\it Analysis of the Laplacian on the complete Riemannian manifold} J. Funct. Anal. 52 (1983), no. 1, 48-79.


\bibitem{Trudinger}  N. Trudinger. {\it  On imbeddings into Orlicz spaces and some applications}, J. Math. Mech. 17 1967 473-483.



\bibitem{YSY} Q. Yang, D. Su, Y. Kong. {\it Sharp Moser-Trudinger inequalities on Riemannian manifolds with negative curvature} Ann. Mat. Pura Appl. (4) 195 (2016), no. 2, 459-471.



\bibitem{Yau}  S.-T. Yau. {\it Harmonic functions on complete Riemannian manifolds}, Comm. Pure Appl. Math. 28 (1975), 201-228.

\end{thebibliography}

\end{document}